\documentclass[final,hidelinks,onefignum,onetabnum]{siamart251216}

\usepackage{amsfonts}
\usepackage{graphicx}
\usepackage{epstopdf}

\ifpdf
  \DeclareGraphicsExtensions{.eps,.pdf,.png,.jpg}
\else
  \DeclareGraphicsExtensions{.eps}
\fi

\usepackage{amsmath, amssymb, stackrel, etoolbox,cite, mathrsfs,bbm}
\usepackage[shortlabels]{enumitem}
\usepackage[utf8]{inputenc}

\usepackage{color,xcolor}
\definecolor{ecvrulecolor}{cmyk}{0.72,0.30,0,0.21} 
\definecolor{ecvsectioncolor}{cmyk}{0.91,0.56,0,0.42} 
\definecolor{ecvhighlightcolor}{cmyk}{0.90,0.28,0,0.20}
\definecolor{ecvtextcolor}{cmyk}{0,0.06,0.10,0.76}

\usepackage{tikz}
\usepackage{tikz-cd}

\usetikzlibrary{arrows,graphs,decorations.markings,decorations.pathmorphing,calc, positioning}
\tikzset{
    admnode/.style={
        minimum height=0.5cm,
        minimum width=2cm
    },
    impl/.style={
        thick,double equal sign distance,->,>=implies
    },
    implback/.style={
        thick,double equal sign distance,<-,>=implies
    },
    impleq/.style={
        thick,double equal sign distance,<->,>=implies
    },
    counterexample/.style={
        red,thick,double equal sign distance,>=implies,degil
    },
    degil/.style={
      decoration={
        markings,
        mark=at position 0.5 with {\node[transform shape] {$/$};}
      },
      postaction={decorate}
    }
}
\tikzset{
    hyp/.style={densely dotted},
    fail/.style={dash pattern=on 2pt off 1.3pt}
}

\newsiamremark{remarks}{Remarks}
\newsiamremark{remark}{Remark}
\newsiamremark{example}{Example}
\newsiamremark{hypothesis}{Hypothesis}
\crefname{hypothesis}{Hypothesis}{Hypotheses}
\newsiamthm{claim}{Claim}
\newsiamremark{fact}{Fact}
\crefname{fact}{Fact}{Facts}
\newsiamremark{setting}{Setting}

\makeatletter
\DeclareOldFontCommand{\rm}{\normalfont\rmfamily}{\mathrm}
\DeclareOldFontCommand{\sf}{\normalfont\sffamily}{\mathsf}
\DeclareOldFontCommand{\tt}{\normalfont\ttfamily}{\mathtt}
\DeclareOldFontCommand{\bf}{\normalfont\bfseries}{\mathbf}
\DeclareOldFontCommand{\it}{\normalfont\itshape}{\mathit}
\DeclareOldFontCommand{\sl}{\normalfont\slshape}{\@nomath\sl}
\DeclareOldFontCommand{\sc}{\normalfont\scshape}{\@nomath\sc}
\makeatother

\usepackage{csquotes}  
\newcommand\q{\enquote}

\DeclareMathOperator*{\esssup}{ess\,sup}

\newcommand \R   {\mathbb{R}}

\newcommand \Kinf{\mathcal{K_\infty}}

\newcommand \KL  {\mathcal{KL}}

\newcommand{\vertiii}[1]{{\left\vert\kern-0.25ex\left\vert\kern-0.25ex\left\vert #1 
    \right\vert\kern-0.25ex\right\vert\kern-0.25ex\right\vert}}

\newcommand \qrq   {\quad\Rightarrow\quad}

\newcommand{\normt}[1]{{\left\vert\kern-0.25ex\left\vert\kern-0.25ex\left\vert #1 
		\right\vert\kern-0.25ex\right\vert\kern-0.25ex\right\vert}}

\usepackage{color,xcolor}
\definecolor{ecvrulecolor}{cmyk}{0.72,0.30,0,0.21} 
\definecolor{ecvsectioncolor}{cmyk}{0.91,0.56,0,0.42} 
\definecolor{ecvhighlightcolor}{cmyk}{0.90,0.28,0,0.20}
\definecolor{ecvtextcolor}{cmyk}{0,0.06,0.10,0.76}

\newif\ifMath					
\newif\ifEngi					

\newif\ifDFGtext					 

\newif\ifAndo              
													
\newif\ifExercises					
\newif\ifSolutions          
\newif\ifGerman							
\newif\ifEnglish						

\newif\ifnothabil						

\newif\ifFuture							

\newif\ifConf                    
\newif\ifJournal								 

\newif\ifNOTFORBOOK
\newif\ifFullVersion
\newif\ifExludedDueToSpaceReasons

\usepackage{xifthen}

\newcommand{\einsnorm}[2]{\ensuremath{
    \!\!\;\!\!\!\;
    \left\bracevert\!\!\!\!\!\left\bracevert
    \!
		\ifthenelse{\isempty{#2}}{#1}{#1(#2)}
    \!
      \right\bracevert\!\!\!\!\!\right\bracevert
    \!\!\;\!\!\!\;
  }}

\usepackage{xcolor}
\definecolor{blond}{rgb}{0.98, 0.94, 0.75}
	
\newlength\mytemplen
\newsavebox\mytempbox

\makeatletter
\newcommand\mybluebox{%
    \@ifnextchar[
       {\@mybluebox}%
       {\@mybluebox[0pt]}}

\def\@mybluebox[#1]{%
    \@ifnextchar[
       {\@@mybluebox[#1]}%
       {\@@mybluebox[#1][0pt]}}

\def\@@mybluebox[#1][#2]#3{
    \sbox\mytempbox{#3}%
    \mytemplen\ht\mytempbox
    \advance\mytemplen #1\relax
    \ht\mytempbox\mytemplen
    \mytemplen\dp\mytempbox
    \advance\mytemplen #2\relax
    \dp\mytempbox\mytemplen
    \colorbox{blond}{\hspace{1em}\usebox{\mytempbox}\hspace{1em}}}

\makeatother

\makeatletter
\let\origd=\d
\renewcommand*\d{
  \relax\ifmmode
    \mathrm{d}%
  \else
    \expandafter\origd
  \fi
}\makeatother

\usepackage{mathtools}

\makeatletter 
\newcommand{\pushright}[1]{\ifmeasuring@#1\else\omit\hfill$\displaystyle#1$\fi\ignorespaces}
\newcommand{\pushleft}[1]{\ifmeasuring@#1\else\omit$\displaystyle#1$\hfill\fi\ignorespaces}
\makeatother 

\newcounter{syscounter}

\newcounter{WPcounter}
\newcounter{PRcounter}

\newcommand{\bbC}{\mathbb{C}}

\newcommand{\bbN}{\mathbb{N}}

\newcommand{\bbR}{\mathbb{R}}

\newcommand{\calK}{\mathcal{K}}
\newcommand{\calL}{\mathcal{L}}

\DeclareMathOperator{\one}{{\mathbbm{1}}} 
\DeclareMathOperator{\re}{Re} 
\newcommand{\argument}{\mathord{\,\cdot\,}} 
\newcommand{\dx}{\;\mathrm{d}} 
\newcommand{\norm}[1]{\left\lVert #1 \right\rVert} 
\newcommand{\modulus}[1]{\left\lvert #1 \right\rvert} 
\newcommand{\duality}[2]{\left\langle#1\, ,\, #2\right\rangle} 
\newcommand{\dom}[1]{\operatorname{dom}\left(#1\right)} 
\DeclareMathOperator{\Ima}{Rg} 

\newcommand{\spec}{\sigma} 
\newcommand{\resSet}{\rho}
\newcommand{\Res}{\mathcal{R}} 

\newcommand{\changeSA}[1]{{\color{black} #1}}
\DeclareRobustCommand{\SICONRunningTitle}{\changeSA{\MakeUppercase{Integral-to-Integral ISS for LTI systems}}}
\headers{\SICONRunningTitle}{S. Arora and A. Mironchenko}

\title{Input-to-state stability in integral norms for linear infinite-dimensional systems
\thanks{\textbf{Funding:} The authors were supported by
German Research Foundation (DFG) grants 523942381 and
MI 1886/3-1, respectively. COST Action 18232 supported
the visit of S. Arora (corresponding author) to A. Mironchenko.}}

\author{Sahiba Arora\thanks{Leibniz Universität Hannover, Institut für Analysis, Welfengarten 1, 30167 Hannover, Germany 
  (\email{sahiba.arora@math.uni-hannover.de}).}
\and Andrii Mironchenko\thanks{Universit\"at Bayreuth, Mathematisches Institut, Universit\"atsstra\ss e 30, 95447 Bayreuth, Germany 
  (\email{andrii.mironchenko@uni-bayreuth.de}).}}

\ifpdf
\hypersetup{
  pdftitle={Integral-to-integral ISS for LTI systems},
  pdfauthor={S. Arora and A. Mironchenko}
}
\fi

\begin{document}

\maketitle

\begin{abstract}
    We study integral-to-integral input-to-state stability for infinite-dimensional linear systems with inputs and trajectories in $L^p$-spaces.
    We start by developing the corresponding admissibility theory for linear systems with unbounded input operators.
    While input-to-state stability is typically characterised by exponential stability and finite-time admissibility, we show that this equivalence does not extend directly to integral norms. For analytic semigroups, we establish a precise characterisation using maximal regularity theory. Additionally, we provide direct Lyapunov theorems and construct Lyapunov functions for $L^p$-$L^q$-ISS and demonstrate the results with examples, including diagonal systems and diffusion equations.
\end{abstract}

\begin{keywords}
    infinite-dimensional systems; linear systems; input-to-state stability; admissibility
\end{keywords}

\begin{MSCcodes}
    93C25, 37B25, 93D09, 93D20, 93C20, 37C75
\end{MSCcodes}

\section{Introduction}

\subsection*{ISS theory} 
Classical input-to-state stability (ISS) gives pointwise bounds on system trajectories at each time instant in terms of the maximum magnitude of the input \cite{Sontag1989, SontagWang1995}. Over time, it became evident that analysing stability with respect to integral norms is crucial, leading E. Sontag in \cite{Sontag1998} to introduce two key notions: integral input-to-state stability (iISS) -- which relates pointwise bounds on the state norm to integral norms of the inputs -- and integral-to-integral input-to-state stability, which examines the relationship between the integral norms of the state and the input.  

Integral ISS theory has been actively developed since its inception in \cite{Sontag1998,AngeliSontagWang2000b,AngeliSontagWang2000a} and has gained prominence in network analysis \cite{Ito2013}; for a comprehensive overview, see \cite{Sontag2008} and \cite[Chapter 4]{Mironchenko2023a}. For infinite-dimensional systems, integral ISS has been investigated in works such as \cite{JacobNabiullinPartingtonSchwenninger2018,Schwenninger2020,MironchenkoIto2015}.

Lyapunov functions play a key role in (integral) ISS analysis of finite- and infinite-dimensional systems \cite{MironchenkoPrieur2020, KarafyllisKrstic2019}. While for finite-dimensional systems, it is natural to use coercive Lyapunov functions, they have significant limitations for the stability analysis of infinite-dimensional systems. Serious obstructions on the way towards constructions of coercive Lyapunov functions for systems whose input operators are not $L^2$-admissible \cite{ZhengZhu2018, EspitiaKarafyllisKrstic2021, JacobMironchenkoPartingtonWirth2020}, have motivated the development of non-coercive Lyapunov theory for non-linear infinite-dimensional systems with disturbances in
\cite{MironchenkoWirth2019}. These results have been generalised to the ISS setting in \cite{JacobMironchenkoPartingtonWirth2020} and applied to the 
ISS analysis of a broad class of infinite-dimensional systems, including boundary control systems. Moreover, ISS Lyapunov methodology has been further developed for analysing ISS with respect to inputs from $L^p$-spaces \cite{Mironchenko2020}.

While both ISS and integral ISS provide pointwise estimates of the system trajectories, integral-to-integral ISS, which relates the integral norms of properly scaled trajectories to the input's integral norms, remains less explored. Sontag showed in \cite{Sontag1998} that for ODE systems with Lipschitz continuous non-linearities, ISS and integral-to-integral ISS are equivalent when the trajectory scaling is appropriately chosen. An infinite-dimensional counterpart of this result was later established in \cite{JacobMironchenkoPartingtonWirth2020}. However, when stability is analysed with respect to specific integral norms of trajectories, the problem ventures largely into uncharted territory.

\subsection*{Admissibility} 
A crucial role in the analysis of infinite-dimensional systems, particularly in the study of well-posed systems, is played by $L^p$-admissible control operators, which facilitate the investigation of stability, controllability, and observability \cite{JacobPartington2004,Staffans2005,TucsnakWeiss2014, TucsnakWeiss2009}. Their importance lies in providing the framework for linking system's inputs to state trajectories, making them an indispensable tool for system analysis. In the Hilbert space setting, admissibility for the case $p=2$ has been extensively studied, yielding a rich and well-developed theory \cite{JacobPartington2001,Staffans2005,TucsnakWeiss2009}. Additionally, significant progress has also been made for other values of $p<\infty$; see \cite{Haak2004,HaakLeMerdy2005,JacobPartingtonPott2014,Weiss1989b}. The case $p=\infty$
and admissibility with respect to continuous inputs
has been the subject of more recent works \cite{JacobNabiullinPartingtonSchwenninger2018,JacobSchwenningerZwart2019, Wintermayr2019, AroraGlueckPaunonenSchwenninger2024, PreusslerSchwenninger2026}.
Notably, the admissibility of control operators is closely linked to the dual concept of admissible observation operators \cite{Weiss1989a, Weiss1989b, AroraSchwenninger2024}.
It is well-known \cite{JacobNabiullinPartingtonSchwenninger2018} that ISS for linear systems with $L^p$-inputs is equivalent to the exponential stability of the undisturbed system (semigroup) combined with finite-time $L^p$-admissibility.
This equivalence allows ISS of linear systems to be analysed using well-established methods from the stability theory of strongly continuous semigroups \cite{vanNeerven1996,EngelNagel2000} and the admissibility theory.

\subsection*{Contribution} 
This paper investigates integral-to-integral input-to-state stability for linear infinite-dimensional systems with inputs and trajectories in $L^p$-spaces. 
Motivated by \cite{JacobNabiullinPartingtonSchwenninger2018}, we introduce the concept of $L^p$-$L^q$-admissibility. As in the ISS case, ISS in integral norms is equivalent to exponential stability of a semigroup together with infinite-time $L^p$-$L^q$-admissibility of a system. 
However, in contrast to the classical admissibility, for exponentially stable systems, infinite-time $L^p$-$L^q$-admissibility is no more equivalent to the (finite-time) $L^p$-$L^q$-admissibility in general, as demonstrated through a counterexample.
However, it turns out that $L^p$-$L^q$-admissibility is closely related to the concept of maximal regularity of Cauchy problems. Maximal regularity is a powerful tool in the theory of non-linear differential equations and has been extensively studied. For an overview, we refer the reader to \cite[Chapter~17]{HytoenenNeervenVeraarWeis2023} or the survey \cite{Dore1993}.
Exploiting the theory of maximal regularity, we establish in Section~\ref{sec:Linear systems and their admissibility} the desired equivalence of finite-time and infinite-time $L^p$-$L^q$-admissibility for analytic semigroups while considering the $L^p$-norms of both the input and the state trajectory. We conclude Section~\ref{sec:Linear systems and their admissibility} by summarising the relationships between various admissibility concepts.  

In Section~\ref{sec:iss-integral}, we proceed to the paper's main topic, $L^p$-$L^q$-input-to-state stability ($L^p$-$L^q$-ISS), and examine its fundamental properties. We explore the relationship between $L^p$-$L^q$-ISS and the admissibility concepts introduced in Section~\ref{sec:Linear systems and their admissibility}, providing insights into their interplay. Additionally, we analyse the special case where $q = p$, highlighting specific characteristics and implications.

Finally, in Section~\ref{sec:lyapunov}, we proceed to a Lyapunov-based analysis of ISS in integral norms. First, we give a characterisation for systems with bounded control operators. We then provide a general sufficient condition for $L^p$-$L^q$-ISS using non-coercive ISS Lyapunov functions. The applicability of these results is demonstrated on diagonal systems and diffusion equations with Dirichlet input. For systems with unbounded control operators, we establish a converse Lyapunov result for analytic semigroups utilising fractional Sobolev spaces.

\subsection*{Notation and Preliminaries}  

Let $\mathbb{R}_+:=[0,\infty)$ be the set of non-negative real numbers and $\one_A$ denote the characteristic function of a set $A$. For normed linear spaces $X$, $Y$, we denote by $\calL(X,Y)$ the space of bounded linear operators between $X$ and $Y$ and use the shorthand $\calL(X):=\calL(X,X)$. The range of a bounded operator $T$ is denoted by $\Ima T$. 
The notations $\dom A$ and $\resSet(A)$ are used for the domain and resolvent set of a closed linear operator $A$ respectively. Also, for $\lambda \in \resSet(A)$, we write $\Res(\lambda,A):=(\lambda-A)^{-1}$ for the resolvent of $A$ at $\lambda$.  For $p\in [1,\infty]$ and a Banach space $U$, we denote by $L^p(\bbR_+,U)$ the space of strongly measurable functions $u:\bbR_+\to U$ with
\[
    \|u\|_{L^p(\bbR_+,U)}:=
    \begin{cases}
        \left(\int_{\bbR_+}\norm{u(s)}^p_U\ ds\right)^{\frac{1}{p}} < \infty, \qquad & \text{if }p<\infty,\\
         \esssup_{s \in\bbR_+}\norm{u(s)}_U < \infty,                           \qquad & \text{if }p=\infty.
    \end{cases}
\]

A $C_0$-semigroup $(T(t))_{t\ge 0}$ with the generator $A$ on a Banach space $X$ is called exponentially stable if there exist $M,\omega>0$ such that $\norm{T(t)}\le Me^{-\omega t}$ for all $t\ge 0$. The extrapolation space $X_{-1}$ associated to $(T(t))_{t\ge 0}$ is defined as the completion of $X$ with respect to the norm 
$
    \norm{x}_{-1}:= \norm{\Res(\lambda, A)x}
$
for a fixed $\lambda \in \resSet(A)$.
The space $X_{-1}$ is a Banach space, and different choices of $\lambda \in \resSet(A)$ generate equivalent norms on $X_{-1}$.
Lifting of the state space $X$ to a larger space $X_{-1}$ brings further good news: the semigroup  $(T(t))_{t\ge 0}$ extends uniquely to a strongly continuous semigroup  $(T_{-1}(t))_{t\ge 0}$ on $X_{-1}$ whose generator $A_{-1}:X_{-1}\to X_{-1}$ is an extension of $A$ with $D(A_{-1}) = X$. For these and other facts about extrapolation spaces, we refer, for instance, to \cite[Section~II.5]{EngelNagel2000} or \cite[Chapter~III]{vanNeerven1992}. 

Further notation and terminology are introduced as and when they're needed.

\section{Linear systems and their admissibility}
    \label{sec:Linear systems and their admissibility}

Throughout this work, we operate under the following key assumptions, which serve as the foundation for our analysis and conclusions:
\begin{setting}
    \label{setting:system}
    Let $\Sigma(A,B)$ be a linear time-invariant system of the form
    \begin{equation}
        \tag{\ensuremath{\Sigma(A,B)}}\label{Sigma(A,B)}
        \begin{split}
            \dot{x}(t) &= Ax(t)+Bu(t),\quad t\ge 0,\\
            x(0)       &= x_0;
        \end{split}
    \end{equation}
    where $A$ generates a $C_0$-semigroup $(T(t))_{t\ge 0}$ on a Banach space $X$ and $B\in \calL(U, X_{-1})$ for some Banach space $U$. 
\end{setting}

For $p\in [1,\infty]$, the \emph{input operator} associated to $\Sigma(A,B)$ is defined as
\[
    \Phi_{\tau}: L^p([0,\tau], U) \to X_{-1}, \quad u \mapsto \int_0^{\tau} T_{-1}(\tau-s) Bu(s)\dx s,
\]
for fixed $\tau\ge 0$. 
As $B\in \calL(U, X_{-1})$, the system $\Sigma(A,B)$ is well-posed in the extrapolation space $X_{-1}$, and its (mild) solution takes the form
\begin{equation}
    \label{eq:flow}
    \phi(\tau,x_0,u)=T(\tau)x_0+\Phi_{\tau} u \in X_{-1}
\end{equation}
for $\tau\ge 0$, $x_0\in X$, and $u \in  L^p([0,\tau], U)$. If the system is well-posed even in $X$, then it is called $L^p$-admissible, i.e., the system $\Sigma(A,B)$ is called \emph{$L^p$-admissible} 
if for some (hence, all) $\tau> 0$,
\[
   u\in L^p([0,\tau], U) \qrq \Phi_{\tau}u \in X.
\]    

\subsection{\texorpdfstring{{\boldmath$L^p$-$L^q$}}{L^p-L^q}-admissibility for linear systems}

If $p<\infty$ and $\Sigma(A,B)$ is $L^p$-admissible, then it is well-known \cite[Proposition~2.3]{Weiss1989b} that the mild solution given by~\eqref{eq:flow} is continuous with values in $X$. In turn,
\[
    \big(u \mapsto \Phi_{(\argument)}u\big)  \in  \calL\Big(L^p([0,t], U), L^q ([0,t], X)\Big).
\]
holds for all $q\in [1,\infty]$ and $t>0$. Without $L^p$-admissibility, we at least always have
\[
    \big(u \mapsto \Phi_{(\argument)}u\big)  \in  \calL\Big(L^p([0,t], U), L^q ([0,t], X_{-1})\Big),
\]
which motivates the following definition:

\begin{definition}
    For $p,q\in [1,\infty]$, the system $\Sigma(A,B)$ in Setting~\ref{setting:system} is called  \emph{(finite-time) $L^p$-$L^q$-admissible} if for some (hence, all) $t> 0$,
    \[
       u\in L^p([0,t], U) \qrq \Phi_{(\argument)}u \in L^q ([0,t], X).
    \]    
    Further, $\Sigma(A,B)$ is called \emph{$L^p$-$\mathrm C$-admissible} if for some (hence, all) $t> 0$,
    \[
        u\in L^p([0,t], U) \qrq   \Phi_{(\argument)}u \in\mathrm C([0,t], X).
    \]
\end{definition}

\begin{remarks}
    \label{rem:admissibility-definition}
    (a) Let $\changeSA{\tau>0}$ be such that  $\Phi_{(\argument)}u \in L^q ([0,\tau], X)$ for all $u\in L^p([0,\tau], U)$. 
    Then for each $t<\tau$, it is easy to see that $\Phi_{(\argument)}u \in L^q ([0,t], X)$ for all $u\in L^p([0,t], U)$ -- by simply extending $u$ by $0$. \changeSA{To see that the same is true for $t=\frac32\tau$, fix
$u\in L^p([0,2\tau],U)$ and choose
$s\in(\frac12\tau,\tau)$ such that $\Phi_su\in X$.
By the composition property of the input maps and local
boundedness of the semigroup, the hypothesis applied to
the shifted input $u(s+\cdot)$ gives
$\Phi_{(\cdot)}u\in L^q([0,s+\tau],X)$, and thus
$\Phi_{(\cdot)}u\in L^q([0,\frac32\tau],X)$
for any $u\in L^p([0,\frac32\tau],U)$.
By induction, we get the latter property for
$t=(\frac32)^n\tau$ as well.} 
    Consequently, $\Phi_{(\argument)}u \in L^q ([0,t], X)$ for all $t>0$ and all $u\in L^p([0,t], U)$.
    
    (b) Note that $\Phi_{(\argument)}u \in L^q ([0,t], X)$ only means that $\Phi_{(\argument)}u$ takes values in $X$ almost everywhere. In particular, \changeSA{for any $u\in L^p([0,t],U)$ there exists $t(u)\in(0,t]$} such that $\Phi_{t(u)}u \in X$ (note that point evaluation of a map $\Phi_{(\argument)}u$ is well-defined as this function is defined everywhere in $X_{-1}$), but this $t(u)$ does not have to be uniform in $u$.
\end{remarks}

The following criterion of $L^p$-$L^q$-admissibility follows by a direct application of the closed graph theorem; cf. \cite[Proposition 4.2]{Weiss1989b}.
\begin{proposition}
    \label{prop:admissibility-criterion}
    Let $p,q\in [1,\infty]$.
    The system $\Sigma(A,B)$ in Setting~\ref{setting:system} is $L^p$-$L^q$-admissible if and only if for each $t>0$ there exists $c(t)>0$ such that
    \begin{align}
    \label{eq:admissibility-uniform}
        \norm{\Phi_{(\argument)}u}_{L^q([0,t], X)} \le c(t) \norm{u}_{L^p([0,t], U)}    
    \end{align}
    for all $u\in L^p([0,t], U)$.    
\end{proposition}

The following strengthening of $L^p$-$L^q$-admissibility is of natural interest owing to Proposition~\ref{prop:admissibility-criterion} and the analogous notion of infinite-time $L^p$-admissibility:

\begin{definition}
    Let $p,q\in [1,\infty]$.
    In Setting~\ref{setting:system}, if $\Sigma(A,B)$ is $L^p$-$L^q$-admissible, and 
    if the constants $c(t)$ in \eqref{eq:admissibility-uniform} can be chosen independently of $t$, i.e., if
    \[
        c_{\infty} := \sup_{t\ge 0} c(t) <\infty,
    \]
    then the system $\Sigma(A,B)$ is called \emph{infinite-time $L^p$-$L^q$-admissible}.
\end{definition}

We establish in the following that for systems with bounded control operator $B$, the exponential stability of the semigroup is sufficient for infinite-time $L^p$-$L^q$-admissibility, provided that $p\le q$. However, this implication does not hold when $p>q$, as demonstrated by a counterexample presented in Example~\ref{ex:Scalar-system-Counterex}.

\begin{proposition}
    \label{prop:admissibility-bounded}
    In Setting~\ref{setting:system}, if $(T(t))_{t\ge 0}$ is exponentially stable and $B\in \calL(U,X)$, then $\Sigma(A,B)$ is infinite-time $L^p$-$L^q$-admissible for all $p,q\in [1,\infty]$ with $p\le q$.
\end{proposition}

\begin{proof}
     By exponential stability, there exist $M,\omega>0$ such that $\norm{T(t)}\le Me^{-\omega t}$ for all $t\ge 0$. Let $p,q\in [1,\infty]$ with $p\le q$. So we can choose $r\in [1,\infty]$ such that $r^{-1}=1-(p^{-1}-q^{-1})$. Fix $t\ge 0$ and let $u\in L^p([0,t], U)$. Let $f= e^{-\omega\argument}\one_{[0,t]}$ and $g$ be the extension of $\norm{Bu(\argument)}$ by $0$ outside $[0,t]$. Then
    \[
        \norm{\Phi_{\tau}u} \le M \int_0^{\tau} e^{-\omega(\tau-s)}\norm{Bu(s)} ds \le M\int_{-\infty}^{+\infty}f(\tau-s)g(s)ds=: M(f\star g)(\tau),
    \]
    for each $\tau \le t$. 

    Young's convolution inequality thereby implies that
    \begin{align*}
        \norm{\Phi_{(\argument)}u}_{L^q([0,t],X)} &\le M\norm{f\star g}_{L^q(\R)} \leq M \norm{f}_{L^r(\R)} \norm{g}_{L^p(\R)}\\
        &\le 
        \begin{cases}
			M(\omega r)^{-1/r}\norm{B}\norm{u}_{L^p([0,t],U)}, & \text{if $r<\infty$},\\
            M\norm{B}\norm{u}_{L^p([0,t],U)}, & \text{if $r=\infty$}, 
		 \end{cases}
    \end{align*}
    and so $\Sigma(A,B)$ is infinite-time $L^p$-$L^q$-admissible.
\end{proof}

\begin{example}
    \label{ex:Scalar-system-Counterex}
    Consider a scalar system
    \begin{align}
        \label{eq:ToySys-LpLq}
        \dot{x} = -x + u.
    \end{align}
    In this case, the semigroup is simply $(e^{-t})_{t\ge 0}$ and the control operator is the identity operator. In particular, the associated input operator is given by
    \[
        \Phi_t u = \int_0^t e^{-(t-s)}u(s)\ ds,
    \]
    from which it is clear that the system is $L^p$-$\mathrm C$ admissible, and thus $L^p$-$L^q$ admissible for each $p,q\in [1,\infty]$. In fact, by Proposition~\ref{prop:admissibility-bounded}, this admissibility is even infinite-time if $p\le q$. 

    Next, fix $p\in (1,\infty)$ and choose $u \in L^p(\bbR_+,\bbR) \setminus L^1(\bbR_+,\bbR)$
    such that $u(t) \geq 0$ for a.e. $t$. Then the solution $\phi(t,0,u)\geq 0$ for all $t\geq 0$.
    Furthermore, by exponential stability, we can apply \cite[Theorem~2.1.21]{Mironchenko2023b} to obtain that \eqref{eq:ToySys-LpLq} is $L^p$-ISS, i.e., there are $M,G>0$ such that the mild solution~\eqref{eq:flow} satisfies
    \[
        \modulus{\phi(t,x,u)}\leq Me^{-t}\norm{x} + G \norm{u}_{L^p([0,t],\bbR)} \qquad (t>0).
    \]

    Since $u\notin L^1(\bbR_+,\bbR)$, and $u$ is nonnegative, it holds that $\int_0^tu(s)\ ds\to\infty$ as $t\to\infty$. 
    On the other hand, because $u \in L^p(\bbR_+,\bbR)$ and $p\ne 1$, we can apply
    \cite[Propositions~4.3 and~4.7]{Mironchenko2023a} with $\xi(r):=r^p$, to obtain that $x(t) = \phi(t,0,u) \to 0$ as $t\to\infty$.
    Therefore, integrating~\eqref{eq:ToySys-LpLq} in the case $x(0)=0$, we have that  
    \begin{align}
        \label{eq:tmp-estimate}
        \int_0^tx(s)ds = -x(t) + \int_0^tu(s)ds \stackrel{t \to \infty}{\longrightarrow}\infty.
    \end{align}
    This means that there is no $c>0$ such that 
    \[
        \norm{\phi(\cdot,0,u)}_{L^1([0,t],\bbR)} \leq c \norm{u}_{L^p([0,t],\bbR)},\quad t\geq 0,
    \]
    and consequently,~\eqref{eq:ToySys-LpLq} is not infinite-time $L^p$-$L^1$ admissible.
\end{example}

\subsection{\texorpdfstring{{\boldmath$L^p$-$L^p$}}{L^p-L^p}-admissibility for analytic semigroups}
\label{sec:LpLp-admissibility for analytic semigroups}

    In the case $p=q$, the notion of $L^p$-$L^q$-admissibility generalises the classical notion of \emph{maximal $L^p$-regularity}. Consider the Cauchy problem
    \begin{subequations}
    \label{eq:max-reg-sys}
    \begin{align}
        \dot{x}(t) & = Ax(t) + f(t), \quad t\in \overline I,\\
        x(0)       & = 0;
    \end{align}        
    \end{subequations}
    where $A$ is a linear operator on a Banach space $X$ and $I=(0,\tau)$ or $I=\bbR_+$. 

\begin{definition}
    Let $p\in [1,\infty]$.
    We say that \emph{$A$ has maximal $L^p$-regularity on $I$} if there exists $C\equiv C(I)\ge 0$ such that for any $f\in L^p(I,X)$, the  Cauchy problem \eqref{eq:max-reg-sys} \changeSA{has a unique strong solution $x$ (in the sense of
\cite[Definition~17.1.1]{HytoenenNeervenVeraarWeis2023}) with $Ax\in L^p(I,X)$ and}
    \[
        \changeSA{\|\dot x\|_{L^p(I,X)}+\|Ax\|_{L^p(I,X)}\le C\|f\|_{L^p(I,X)}.}
    \]
\end{definition}

\begin{remark}
    \label{rem:maximal-regularity}
    Maximal regularity ensures that the operator $A$ generates an analytic $C_0$-semigroup $(T(t))_{t\ge 0}$ on $X$; see \cite[Theorem~17.2.15]{HytoenenNeervenVeraarWeis2023}. Moreover, we know from \cite[Theorem~17.2.19]{HytoenenNeervenVeraarWeis2023} that if $A$ generates an analytic $C_0$-semigroup $(T(t))_{t\ge 0}$ on $X$, then 
    $A$ has maximal $L^p$-regularity on $I$ if and only if there exists $C>0$ such that
    \begin{align}
    \label{eq:max-reg-reformulation}
    \hspace{-3mm}\norm{\int_0^{\argument} \hspace{-2mm}\changeSA{T_{-1}}(\argument-s)A_{-1}f(s)\ ds}_{L^p(I, X)} \hspace{-3mm}= \norm{A\int_0^{\argument} \hspace{-2mm} T(\argument-s)f(s)\ ds}_{L^p(I, X)} \hspace{-3mm}\le C\norm{f}_{L^p(I, X)}    
    \end{align}
    \changeSA{for all $f\in L^p(I,X)$. Here the first integral is taken
in $X_{-1}$, and we use that
$\int_0^{\cdot}A_{-1}T(\cdot-s)f(s)\,ds
=A\int_0^{\cdot}T(\cdot-s)f(s)\,ds$
a.e. on $I$.} In other words, for generators of analytic semigroups, on bounded intervals, maximal $L^p$-regularity is equivalent to $L^p$-$L^p$-admissibility of $\Sigma(A,A_{-1})$ and if $I=\bbR_+$, then maximal $L^p$-regularity is equivalent to infinite-time $L^p$-$L^p$-admissibility of $\Sigma(A, A_{-1})$. 
\end{remark}

In view of Remark~\ref{rem:maximal-regularity}, we adapt some well-known results about maximal regularity to our setting  to obtain results about $L^p$-$L^p$-admissibility. To this end, we use the following observation:~suppose that $0\in \resSet(A)$, so that $A^{-1} \in \calL(X,D(A))$ and in turn, $A^{-1}_{-1} \in \calL(X_{-1},X)$. Since $B\in \calL(U,X_{-1})$, we infer $A_{-1}^{-1}B \in \calL(U,X)$. \changeSA{Thus for each locally integrable $u:\bbR_+\to U$,} we obtain that $f:=A_{-1}^{-1}Bu$ maps into $X$. Therefore for each $t$, we can write
    \begin{equation}
        \label{eq:input-with-f}
        \Phi_t u = \int_0^{t} T_{-1}(t-s)A_{-1}A_{-1}^{-1}Bu(s)\ ds
                = A_{-1}\int_0^{t} T(t-s)f(s)\ ds,
    \end{equation}
 where we have used the facts that the generator of the semigroup commutes with the semigroup operators and the extrapolated semigroup agrees with the original semigroup on $X$.

 \begin{proposition}
    \label{prop:maximal-regularity-implies-admissibility}
    The following  are equivalent for an analytic $C_0$-semigroup $(T(t))_{t\ge 0}$ generated by $A$ on a Banach space $X$ and $p\in [1,\infty]$.
    \begin{enumerate}[\upshape (i)]
        \item For each Banach space $U$ and each $B\in \calL(U, X_{-1})$, the system $\Sigma(A,B)$ is $L^p$-$L^p$-admissible.
        \item The system $\Sigma(A,A_{-1})$ is $L^p$-$L^p$-admissible.
        \item The operator $A$ satisfies maximal $L^p$-regularity on bounded intervals.
    \end{enumerate}
    Moreover, \changeSA{if $0\in\rho(A)$,} the analogous equivalence holds between infinite-time $L^p$-$L^p$-admissibility and maximal $L^p$-regularity on $\bbR_+$.
\end{proposition}

\begin{proof}
    Due to Remark~\ref{rem:maximal-regularity}, the assertions (ii) and (iii) are equivalent and implication (i) $\Rightarrow$ (ii) is obvious.

    ``(iii) $\Rightarrow$ (i)'': Suppose that $A$ satisfies maximal $L^p$-regularity on $I$ and let $u \in L^p(I,U)$; where $I=(0,\tau)$ or $I=\bbR_+$. 
    \changeSA{If $I$ is bounded, translating the generator,} we may assume without loss of generality that $0 \in \resSet(A)$, so that $f:=A_{-1}^{-1}Bu\in L^p(I,X)$. By maximal regularity,~\eqref{eq:input-with-f}, and \eqref{eq:max-reg-reformulation}, there exists $C_{I}>0$ such that
    \[
        \changeSA{\norm{\Phi_{(\argument)}u}_{L^p(I, X)}}\le C_{I}\norm{f}_{L^p(I, X)} \le C_{I}\norm{A_{-1}^{-1}B}_{\calL(U, X)} \norm{u}_{L^p(I, U)}
    \]
    and so $\Sigma(A,B)$ is (in)finite-time $L^p$-$L^p$-admissible (if $I=\bbR_+$).
\end{proof}

Recall that an analytic semigroup is called \emph{bounded analytic} if its extension to a certain sector is bounded, see \cite[Section~II.4.5]{EngelNagel2000} for a precise definition. 
As mentioned before, maximal regularity of $A$ on (un)bounded intervals implies that $A$ generates (bounded) analytic semigroup. 
A remarkable result due to de Simon shows that the converse is true in Hilbert spaces; see \cite[Theorem~4.1]{Dore1993} or \cite[Corollary~17.3.8]{HytoenenNeervenVeraarWeis2023}. Combining \cite[Corollary~17.3.8]{HytoenenNeervenVeraarWeis2023} with Proposition~\ref{prop:maximal-regularity-implies-admissibility}, we immediately obtain the following:

\begin{corollary}
    \label{cor:hilbert-analytic-admissible}
    Let $(T(t))_{t \ge 0}$ be a $C_0$-semigroup on a Hilbert space $X$.
    \begin{enumerate}[\upshape (a)]
        \item If the semigroup $(T(t))_{t \ge 0}$ is bounded analytic \changeSA{and $0\in\rho(A)$}, then $\Sigma(A,B)$ is infinite-time $L^p$-$L^p$-admissible for all $p\in (1,\infty)$ and all $B\in \calL(U, X_{-1})$.

        \item If the semigroup $(T(t))_{t \ge 0}$ is analytic, then $\Sigma(A,B)$ is finite-time $L^p$-$L^p$-admissible for all $p\in (1,\infty)$ and all $B\in \calL(U, X_{-1})$.
    \end{enumerate}
\end{corollary}

\begin{remark}
    Let $(T(t))_{t \ge 0}$ be an exponentially stable $C_0$-semigroup on a Hilbert space $X$. If $(T(t))_{t \ge 0}$ is bounded analytic and similar to a contraction semigroup, then $\Sigma(A,B)$ is infinite-time $L^\infty$-$\mathrm C$-admissible for all $B\in \calL(U, X_{-1})$ with $\dim U <\infty$. Indeed, in this case, we know from \cite[Theorem~1]{JacobSchwenningerZwart2019} that $\Sigma(A,B)$ is finite-time $L^\infty$-admissible, and all solutions corresponding to inputs in $L^\infty(\R_+,U)$ - are continuous. As $T$ is exponentially stable, this means that $\Sigma(A,B)$ is also infinite-time $L^\infty$-admissible, and thus $\Sigma(A,B)$ is  $L^\infty$-ISS, which implies that all solutions corresponding to $L^\infty(\R_+,U)$ are continuous and globally bounded, thus 
    $\Sigma(A,B)$ is infinite-time $L^\infty$-$\mathrm C$-admissible.
\end{remark}

For exponentially stable semigroups, $L^p$-admissibility is equivalent to infinite-time $L^p$-admissibility (see \cite[Lemma~2.9]{JacobNabiullinPartingtonSchwenninger2018} or \cite[Lemma~2.1.20]{Mironchenko2023b}). For the case of $L^p$-$L^q$-admissibility, the situation is more complex; see Example~\ref{ex:Scalar-system-Counterex}.
However, if the semigroup is analytic, then we have the following result, motivated by \cite[Theorem~2.4]{Dore1993}:

\begin{theorem}
    \label{thm:infinite-admissibility-sufficient}
    In Setting~\ref{setting:system}, let $(T(t))_{t\ge 0}$ be exponentially stable and analytic, and let $p\in [1,\infty]$. Then $\Sigma(A,B)$ is finite-time $L^p$-$L^p$-admissible if and only if $\Sigma(A,B)$ is infinite-time $L^p$-$L^p$-admissible.
\end{theorem}

\begin{proof}
     Let $\Sigma(A,B)$ be $L^p$-$L^p$-admissible for some 
     \changeSA{$p\in[1,\infty]$} and pick any $u\in L^p(\bbR_+,U)$.
    Firstly, exponential stability and analyticity of the semigroup imply that there exist $M,\omega>0$ such that
    \[
        \norm{AT(t)} = \norm{AT(1)T(t-1)} \le Me^{-\omega(t-1)},\qquad t\ge 1;
    \]
    here we have used that $\norm{AT(\argument)}\le C_1(\argument)^{-1}$ for some $C_1>0$, which is true by \cite[Theorem~2.6.13]{Pazy1983} because the semigroup is analytic and $0$ is in the resolvent set of $A$ (owing to exponential stability of the semigroup). Taking $f:=A_{-1}^{-1}Bu\in L^p(\bbR_+,X)$, we therefore obtain, that
    \[
        t\mapsto \int_0^{t-1} \changeSA{T_{-1}}(t-s)Bu(s)\ ds = \int_0^{t-1} AT(t-s)f(s)\ ds,\qquad t\geq 1
    \]
    lies in $L^p([1,\infty),X)$; for $p=\infty$ this is obvious, and for $p<\infty$ this follows by Young's convolution inequality \cite[Lemma~1.2.30]{HytoenenNeervenVeraarWeis2016}.
    Combining this observation with~\eqref{eq:input-with-f}, it only remains to show that there exists $C>0$ such that
    \begin{equation}
        \label{eq:tail}
        \norm{t\mapsto\int_{t-1}^t AT(t-s)f(s)\ ds}_{L^p([1,\infty), X) } \le C \norm{u}_{L^p(\bbR_+, U)};
    \end{equation}
    see Proposition~\ref{prop:admissibility-criterion}.
    We divide the proof of this into several steps.
    \changeSA{Extend $u$ and $f$ to $\bbR$ by zero and set
$u_{-1}=f_{-1}=0$.}
    
    \emph{Step 1}: For each $k\in \bbN_0\changeSA{:=\bbN\cup\{0\}}$, define $f_k=\one_{[k,k+1]}f$ and $u_k=\one_{[k,k+1]}u$. 
    For each $t\in [k,k+1]$, we write 
    \begin{align*}
        \int_{t-1}^t AT(t-s)f(s)\ ds & = \int_{t-1}^k AT(t-s)f(s)\ ds + \int_{k}^t AT(t-s)f(s)\ ds\\
                                     & = \int_{t-1}^k AT(t-s)f_{k-1}(s)\ ds + \int_0^{t-k} \hspace{-3mm} A T(t-k-s)f_k(s+k)\ ds\\
                                     & =: a_k(t) + b_k(t).                           
    \end{align*}

    \emph{Step 2}: Using the 
    relation $\changeSA{A_{-1}}f_k(s+k) = Bu_k(s+k)$, $s \in[0,t-k]$
    and~\eqref{eq:input-with-f}, we obtain
    $
        b_k(t)  = \Phi_{t-k}\big( u_k(\argument +k) \big).
    $
    By finite-time $L^p$-$L^p$-admissibility, there exists $C_2>0$ such that 
    \begin{equation}
        \label{eq:admissibility-estimate}
        \begin{aligned}
            \norm{b_k}_{L^p( [k,k+1], X) } &= \norm{t\mapsto \Phi_{t}u_k(\argument +k)}_{L^p([0,1], X)}\\
                                    &\le C_2 \norm{u_k(\argument +k)}_{L^p([0,1], U)}
                                           \le C_2 \norm{u_k}_{L^p([k,k+1], U)}.
        \end{aligned}
    \end{equation}

    \changeSA{\emph{Step 3}: For $k\ge1$ and $t\in[k,k+1]$, set
$w_k(t):=\int_{k-1}^{k}T_{-1}(t-s)Bu_{k-1}(s)\,ds$
and $r_k(t):=\int_{k-1}^{t-1}AT(t-s)f_{k-1}(s)\,ds$.
By finite-time $L^p$-$L^p$-admissibility, there exists
$C_3>0$ such that
$\|w_k\|_{L^p([k,k+1],X)}
\le C_3\|u_{k-1}\|_{L^p([k-1,k],U)}$.
Using $a_k=w_k-r_k$ and
$f_{k-1}=A_{-1}^{-1}Bu_{k-1}$, we get
\begin{equation}
    \begin{aligned}
        \|a_k\|_{L^p([k,k+1],X)}
        &\le C_3\|u_{k-1}\|_{L^p([k-1,k],U)}
        +\sup_{s\in[1,2]}\|AT(s)\|
         \|f_{k-1}\|_{L^p([k-1,k],X)}\\
        &\le C_4\|u_{k-1}\|_{L^p([k-1,k],U)};
    \end{aligned}
    \label{eq:integrable-kernel}
\end{equation}
where $C_4:=C_3+\sup_{s\in[1,2]}\|AT(s)\| \|A_{-1}^{-1}B\|$.}

    \emph{Step 4}:
    Suppose that $p<\infty$. Then using $(a+b)^p \le 2^{p-1} (a^p+b^p)$, we obtain from Step~1 that
    \[
        \norm{ \int_{t-1}^t AT(t-s)f(s)\ ds }^p \le 2^{p-1}\big( \norm{a_k(t)}^p + \norm{b_k(t)}^p\big).
    \]
    Plugging in the estimates~\eqref{eq:integrable-kernel} and~\eqref{eq:admissibility-estimate}, we obtain for each $k\in \bbN_0$ that
    \begin{align*}
        \int_k^{ k+1} \norm{ \int_{t-1}^t AT(t-s)f(s)\ ds }^p \ dt 
                    & \le 2^{p-1} \big( \norm{a_k}_{L^p( [k,k+1], X) }^p + \norm{b_k}_{L^p( [k,k+1], X)}^p \big) \\
                    & \hspace{-12mm}\le 2^{p-1} \big( C_4^p \norm{ u_{k-1}}_{L^p( [k-1,k], U) }^p + C_2^p \norm{u_k}_{L^p([k,k+1], U)}^p\big).
    \end{align*}
    Consequently, there exists $C'>0$ such that 
    \[
        \norm{t\mapsto\int_{t-1}^t AT(t-s)f(s)\ ds}_{L^p([1,\infty), X) }^p \hspace{-3mm} \le C' \sum_{k\in \bbN_0 } \norm{u_k}_{L^p([k,k+1], U)}^p = C' \norm{u}_{L^p(\bbR_+, U)}^p.
    \]
    On the other hand, the estimates~\eqref{eq:integrable-kernel} and~\eqref{eq:admissibility-estimate} together with Step~1 give
    \begin{align*}
        \norm{t\mapsto\int_{t-1}^t AT(t-s)f(s)\ ds}_{L^\infty([1,\infty), X) }
                    & \hspace{-4mm}\le \sup_{k\in\bbN}\big(\norm{a_k}_{L^\infty( [k,k+1], X) } {+} \norm{b_k}_{L^\infty( [k,k+1], X)}\big)\\
                    & \hspace{-17mm}\le \sup_{k\in\bbN}\big( C_4\norm{u_{k-1}}_{L^\infty( [k-1,k], U) } + C_2\norm{u_k}_{L^\infty( [k,k+1], U)}\big)\\
                    & \hspace{-17mm} \le (C_2+C_4) \norm{u}_{L^\infty(\bbR_+, U)}.
    \end{align*}
    In either case, there exists $C>0$ such that~\eqref{eq:tail} holds.
\end{proof}

For evolution equations, maximal $L^p$-regularity is independent of $p\in (1,\infty)$; see \cite[Theorem~4.2]{Dore1993} or \cite[Theorem~17.2.31]{HytoenenNeervenVeraarWeis2023}. For analytic semigroups, the proof can be generalised to our setting:

\begin{proposition}
    \label{prop:p-implies-q}
    In Setting~\ref{setting:system}, suppose that $(T(t))_{t\ge 0}$ is analytic. If $\Sigma(A,B)$ is \changeSA{finite-time} $L^p$-$L^p$-admissible for some $p\in [1,\infty]$, then it is even \changeSA{finite-time} $L^q$-$L^q$-admissible for all $q\in (1,\infty)$.
\changeSA{Moreover, if $(T(t))_{t\ge0}$ is exponentially stable,
the analogous infinite-time equivalence holds.}
\end{proposition}

The proof of Proposition~\ref{prop:p-implies-q} requires the following simple observation:

\begin{remark}
    \label{rem:admissibility-translation}
    Note that for $\omega \in \bbC$, the semigroup generated by $A+\omega$ is given by $(e^{\omega t} T(t))_{t\ge 0}$. Using this, it is easy to check for $p\in [1,\infty]$ that $\Sigma(A,B)$ is finite-time $L^p$-$L^p$-admissible if and only if $\Sigma(A+\omega,B)$ is finite-time $L^p$-$L^p$-admissible.
\end{remark}

\begin{proof}[Proof of Proposition~\ref{prop:p-implies-q}]
    We outline how the proof of \cite[Theorem~17.2.31]{HytoenenNeervenVeraarWeis2023} can be adapted to our setting:~
    \changeSA{Assume first that $(T(t))_{t\ge0}$ is exponentially stable
and $\Sigma(A,B)$} is infinite-time $L^p$-$L^p$-admissible for some $p\in [1,\infty]$ and let $q\in (1,\infty)$. \changeSA{In particular, $0\in\rho(A)$.} As in~\eqref{eq:input-with-f}, the corresponding input operator is defined by the convolution with kernel
    \[
        \changeSA{K:=\one_{(0,\infty)}ATA_{-1}^{-1}B.}
    \]
    
    Infinite-time $L^p$-$L^p$-admissibility guarantees that the input operator is bounded from $L^p(\bbR_+, U)$ to  $L^p(\bbR_+, X)$. 
    On the other hand, using analyticity \changeSA{and exponential stability} of the semigroup and repeating the 
    arguments of \cite[Lemma~17.2.30]{HytoenenNeervenVeraarWeis2023}, one can show that $K$ satisfies a Hörmander type estimate and in turn -- employing a multiplier theorem \cite[Theorem~11.2.5]{HytoenenNeervenVeraarWeis2023} -- the input operator is bounded from $L^q(\bbR_+,U)$ to $L^q(\bbR_+,X)$.  It follows that $\Sigma(A,B)$ is infinite-time $L^q$-$L^q$-admissible.

    \changeSA{Next, consider the case that $\Sigma(A,B)$ is finite-time
$L^p$-$L^p$-admissible. Owing to Remark~\ref{rem:admissibility-translation} and
Theorem~\ref{thm:infinite-admissibility-sufficient}, for $\omega>\max\{0,\omega_0(A)\}$
the system $\Sigma(A-\omega,B)$ is infinite-time
$L^p$-$L^p$-admissible (here $\omega_0(A)$ denotes the
growth bound of the semigroup generated by $A$).
Therefore, our previous arguments show that
$\Sigma(A-\omega,B)$ is even infinite-time
$L^q$-$L^q$-admissible. Once again, appealing to
Remark~\ref{rem:admissibility-translation}, we obtain that $\Sigma(A,B)$ is finite-time
$L^q$-$L^q$-admissible.}
\end{proof}

We point out that \cite[Theorem~II.1.3(a)]{RubioRuizTorrea1986} applied above allows the extension of the so-called \emph{singular integral operators (of convolution type)} -- see \cite[Definition~II.1.2]{RubioRuizTorrea1986} for a definition -- between $L^p$ spaces for $p\in [1,\infty]$ to bounded operators between $L^q$ spaces for $q\in (1,\infty)$.

In Proposition~\ref{prop:p-implies-q}, (in)finite-time $L^1$-$L^1$-admissibility is stronger than  the corresponding $L^q$-$L^q$-admissibility for $q\in (1,\infty)$. This can be seen from the corresponding counterexamples for maximal regularity \cite[Examples~17.4.6 and~17.4.7]{HytoenenNeervenVeraarWeis2023} owing to Proposition~\ref{prop:maximal-regularity-implies-admissibility}; see also the theorem of Guerre-Delabriere \cite[Corollary~17.4.5]{HytoenenNeervenVeraarWeis2023}.

While $L^p$-$L^p$-admissibility implies $L^p$-$L^q$-admissibility can be checked easily for $q\le p$, the implication may fail for $q>p$, as the following example shows:

\begin{example}
    \label{ex:LpLp-adm_but_notLpLinf-adm}
    On the Hilbert space 
    \begin{align}
    \label{eq:ell2}
        X:=\ell^2(\bbN)=\left\{(x_k)_{k=1}^\infty \changeSA{\in\bbC^{\bbN}}:\ \sum_{k=1}^\infty |x_k|^2 <\infty\right\},
    \end{align}
    consider the diagonal semigroup $(T(t))_{t \ge 0}$ generated by
    \[
        Ae_n = -n e_n \quad \text{with}\quad \dom{A} := \left\{ (x_n)\in X: (n x_n) \in X\right\}.
    \]
    \changeSA{Here $(e_n)_{n\in\bbN}$ is the canonical orthonormal
basis in $\ell^2$.}
    We show that the system $\Sigma(A, A_{-1})$
    \begin{enumerate}[\upshape (a)]
        \item is $L^p$-$L^p$-admissible for all $p\in (1,\infty)$ but
        \item not $L^p$-$L^\infty$-admissible for any $p\in [1,\infty]$.
    \end{enumerate}
    
    (a) Since $A$ is the generator of a bounded analytic semigroup on a Hilbert space, for each $B\in \calL(X, X_{-1})$ (in particular, for $B=A_{-1}$) and $p\in (1,\infty)$, the system $\Sigma(A,B)$
    is $L^p$-$L^p$-admissible by Corollary~\ref{cor:hilbert-analytic-admissible}. 
    
    \changeSA{(b) Since $L^p$-admissibility of $B\in\calL(U,X_{-1})$
for $p<\infty$ implies zero-class $L^q$-admissibility
for all $q>p$, by \cite[Proposition~1.3 and Theorem~2.9]{JacobSchwenningerWintermayr2022},
the control operator $A_{-1}$ is not $L^p$-admissible
for any $p\in[1,\infty]$. By reflexivity of $X$, it follows
that $\Sigma(A,A_{-1})$ is not $L^p$-$L^\infty$-admissible
for any $p\in[1,\infty]$.}
\end{example}

\begin{example}
    \label{ex:LpLinf-adm_but_notLpC-adm}
    Consider the same semigroup as in Example~\ref{ex:LpLp-adm_but_notLpLinf-adm} above but instead with the state space $X=c_0(\bbN)$. In this case, the semigroup has $L^\infty$-maximal-regularity but $B=A_{-1}$ is not $L^\infty$-admissible; see \cite[Example~2.3]{JacobSchwenningerWintermayr2022} and Remark~\ref{rem:maximal-regularity}. In turn, $\Sigma(A,A_{-1})$ is $L^\infty$-$L^\infty$-admissible but not $L^\infty$-$\mathrm C$-admissible.
\end{example}

To wrap up this section, we provide a summary of the connections among different admissibility concepts:

\begin{theorem}
    \label{thm:Simple-relationships}
    The relationships in Figure~\ref{fig:Admissibility-Relations} hold.
\end{theorem}

\begin{proof}
    The downward implications on the left are easy to see.
    We mention here the other implications:
    \begin{itemize}
        \item It is well-known that infinite-time $L^p$-admissibility for exponentially stable semigroups is equivalent to $L^p$-admissibility; see \cite[Lemma~2.9]{JacobNabiullinPartingtonSchwenninger2018} or (for $p=2$ in a dual form) \cite[Lemma~1.1]{Gra95}.
        
        \item Suppose that $\Sigma(A,B)$ is $L^p$-admissible for some $p<\infty$. Then by \cite[Remark~2.6]{Weiss1989b}, we have $\phi(\argument,x,u) \in\mathrm C ([0,t], X)$ for all $u\in L^p([0,t], U)$. Thus, $\Sigma(A,B)$ is $L^p$-$\mathrm C$-admissible.
                
    \item Theorem~\ref{thm:infinite-admissibility-sufficient} shows that if $(T(t))_{t\ge 0}$ is an exponentially stable analytic semigroup, then $\Sigma(A,B)$ is finite-time $L^p$-$L^p$-admissible $\Leftrightarrow$ $\Sigma(A,B)$ is infinite-time $L^p$-$L^p$-admissible.
    \end{itemize}
\end{proof}

\begin{figure}
    \centering
    \resizebox{\linewidth}{!}{%
    \begin{tikzpicture}[admnode]
        \coordinate (A) at (0,0);
        \coordinate (B) at (7.8,0);
        \coordinate (B2) at (7.8,1.5);
        \coordinate (C) at (0,-2.2);
        \coordinate (D) at (9.2,-2.2);
        \coordinate (E) at (0,-4.4);
        
        \node (LpC) at (A) {$L^p$-$\mathrm C$-admissibility};
        \node (LpAdm) at (B) {$L^p$-admissibility};
        
        \node[text centered,text width=8cm] at (B2)
        {infinite-time $L^p$-admissibility};
        
        \node (LpLq) at (C) {$L^p$-$L^q$-admissibility};
        
        \node (InfLpLq) [text centered,text width=4cm] at (D)
        {infinite-time\\ $L^p$-$L^p$-admissibility};
        
        \node (DualLpLq) at (E)
        {$L^{p'}$-$L^{q'}$-admissibility};
        
        \draw[implback] ($(B)+(-0.2,0.2)$) -- ++(0,0.95);
        \draw[ecvsectioncolor,impl,hyp]     ($(B)+(0.2,0.2)$)  -- ++(0,0.95);
        
        \node[ecvsectioncolor,text width=4cm] at ($(B)+(2.6,0.75)$)
        {\small $T$ is exp.\ stable\\ \hspace{-1mm}\cite[Lemma~2.9] {JacobNabiullinPartingtonSchwenninger2018}};
        
        \draw[ecvsectioncolor,impl,hyp] ($(B)+(-1.8,0.2)$) -- ($(A)+(1.9,0.2)$);
        \draw[implback]             ($(B)+(-1.8,-0.2)$) -- ($(A)+(1.9,-0.2)$);
        
        \node[ecvsectioncolor] at ($(A)!0.5!(B)+(0,0.6)$)
        {\small $p\in[1,\infty)$, by \cite[Remark~2.6]{Weiss1989b}};
        
        \draw[impl] (LpC) -- (LpLq);
        \draw[ecvsectioncolor,impl,hyp] (LpLq) -- (DualLpLq);
        
        \draw[ecvsectioncolor,impleq,hyp]
        (LpLq) -- ($(D)+(-1.6,0)$)
        node[midway,above,sloped]
        {\small $T$ is exp.\ stable + analytic, $q{=}p \in [1,\infty]$}
        node[midway,below,sloped]
        {\small Theorem~\ref{thm:infinite-admissibility-sufficient}};
        
        \draw[counterexample,->]
        ($(C)+(0.4,0.3)$) -- ($(A)+(0.4,-0.3)$);
        
        \node[red] at ($(A)!0.5!(C)+(1.6,0)$)
        {\small Example~\ref{ex:LpLinf-adm_but_notLpC-adm}};
        
        \draw[counterexample,<-,degil]
        ($(C)+(0.4,-0.3)$) -- ($(E)+(0.4,0.35)$);
        
        \node[red] at ($(C)!0.5!(E)+(1.6,0)$)
        {\small Example~\ref{ex:LpLp-adm_but_notLpLinf-adm}};
        
        \node[ecvsectioncolor,align=left]
        at ($(C)!0.5!(E)+(-1,0)$)
        {\small $p' \in [p,\infty]$,\\ $q'\in [1,q]$};
    \end{tikzpicture}%
}
    \caption{Relations between admissibility notions for fixed $p,q\in[1,\infty]$.
    Solid arrows are unconditional; dotted arrows use the indicated additional assumptions; slashed arrows indicate failure in general.}
    \label{fig:Admissibility-Relations}
\end{figure}

\section{Input-to-state stability in integral norms}
    \label{sec:iss-integral}

Having introduced the admissibility concepts, we proceed to the stability and robustness analysis. 
We start with the classical input-to-state stability concept for systems with integrable inputs. Recall the following classes of comparison functions:
\begin{align*}
     \calK &:= \{ \gamma: \bbR_+ \to \bbR_+ : \gamma\text{ is continuous, }\gamma(0)=0,\text{ and strictly increasing}\},\\
    \calL &:= \{ \gamma: \bbR_+ \to \bbR_+ : \gamma\text{ is continuous, strictly decreasing, and }\lim_{t\to\infty}\gamma(t)=0\},\\
    \calK_{\infty} &:= \{\gamma \in \calK: \gamma\text{ is unbounded}\}, \\
    \calK\calL &:= \{ \beta: \bbR_+^2 \to \bbR_+ : \beta\text{ is continuous, }\beta(\argument,t)\in \calK, \text{ and }\beta(r,\argument)\in \calL~\forall t\ge 0,r> 0\}.
\end{align*}
We refer to \cite{Kellett2014} and \cite[Appendix A]{Mironchenko2023a} for an overview of the properties of comparison functions.

\begin{definition}
    \label{def:LpISS}
    For $p\in [1,\infty]$, we say that $\Sigma(A,B)$ in Setting~\ref{setting:system} is \emph{$L^p$-input-to-state stable ($L^p$-ISS)} if it is $L^p$-admissible, and there are $\beta \in \KL$ and $\gamma\in \calK_{\infty}$ such that its mild solution given by~\eqref{eq:flow} for all $x\in X$ and $u\in L^p([0,t], U)$ satisfies
    \begin{equation}
        \tag{\ensuremath{L^p\text{-ISS}}}\label{eq:p-iss}
        \norm{\phi(t,x,u)}
        \le
        \beta(\norm{x},t)
        +
        \gamma\left(\norm{u}_{L^p([0,t],U)}\right)
        \quad \forall\, t\ge 0.
    \end{equation}
\end{definition}

Inspired by integral-to-integral stability concepts introduced by Sontag \cite{Sontag1998}, we want to study the following notion:
\begin{definition}
    \label{def:LpLqISS}
    For $p,q\in [1,\infty]$, we say that
    $\Sigma(A,B)$ in Setting~\ref{setting:system} 
    is
     \emph{$L^p$-$L^q$-input-to-state stable ($L^p$-$L^q$-ISS)} if it is $L^p$-$L^q$-admissible, and there are $\sigma,\gamma\in \calK_{\infty}$ such that the mild solution given by~\eqref{eq:flow} satisfies
    \begin{equation}
        \tag{\ensuremath{L^p\text{-}L^q\text{-ISS}}}
        \label{eq:p-q-iss}
         \norm{\phi(\argument,x,u)}_{L^q([0,t], X)} \le \sigma(\norm{x})+ \gamma\left(\norm{u}_{L^p([0,t], U)}\right) \quad \forall\ t\ge 0
    \end{equation}
    for all $x\in X$ and $u\in L^p([0,t], U)$.
\end{definition}

Since $L^p$-ISS is a pointwise estimate, $L^p$-ISS systems are automatically $L^p$-$L^\infty$-ISS:
\begin{proposition}
    \label{prop:p-infty-iss}
    In Setting~\ref{setting:system}, if the system $\Sigma(A,B)$ is $L^p$-ISS for some $p\in [1,\infty]$, then $\Sigma(A,B)$ is $L^p$-$L^\infty$-ISS. 
\end{proposition}

\begin{proof}
    Since $\Sigma(A,B)$ is $L^p$-ISS, and $\beta \in\KL$,
    \[
         \norm{\phi(t,x,u)} \le \beta(\norm{x},0) + \gamma\left(\norm{u}_{L^p([0,t], U)}\right) \quad \forall\ t\ge 0,
    \]
    and taking the essential supremum over $[0,t]$, \changeSA{we obtain the claim with
$\sigma(r):=\beta(r,0)+r\in\mathcal K_\infty$.}
\end{proof}
Note that $L^p$-$L^\infty$-ISS does not imply $L^p$-ISS (or even global attractivity for zero inputs), already for $A=B=0$.

Whereas $L^p$-$L^\infty$-ISS follows immediately from the classical $L^p$-ISS, for $q<\infty$ we have the following characterisation of $L^p$-$L^q$-ISS inspired by the analogous characterisation of $L^p$-ISS in \cite[Theorem~2.1.21]{Mironchenko2023b}.

\begin{proposition}
    \label{prop:p-q-iss-implies-exponential-stability}
    In Setting~\ref{setting:system}, the following assertions are equivalent for $p\in [1,\infty]$ and $q\in [1,\infty)$.
    \begin{enumerate}[\upshape (i)]
        \item The system $\Sigma(A,B)$ is $L^p$-$L^q$-ISS.
        \item The system $\Sigma(A,B)$ is infinite-time $L^p$-$L^q$-admissible and $(T(t))_{t \ge 0}$ is exponentially stable.
        \item There exist $M, G>0$ such that
            \[
                \norm{\phi(\argument,x,u)}_{L^q([0,t], X)} \le M\norm{x} + G \norm{u}_{L^p([0,t], U)}
            \]
            for all $t\ge 0, x\in X$, and $u\in L^p([0,t], U)$.
    \end{enumerate}
\end{proposition}

\begin{proof}
    Of course, (iii) implies (i), whereas the implication (ii) $\Rightarrow$ (iii) is a simple application of the triangle inequality.

    ``(i) $\Rightarrow$ (iii)'':
    First of all, let $\Sigma(A,B)$ be $L^p$-$L^q$-ISS. Setting $u=0$ in~\eqref{eq:p-q-iss}, we obtain
    \[
        \changeSA{\int_0^\infty\|T(s)x\|^q\,ds<\infty\qquad\forall x\in X.}
    \]
    and so $(T(t))_{t\ge 0}$ is exponentially stable by the theorem of Datko-Pazy \cite[Theorem~V.1.8]{EngelNagel2000}; here we have used $q<\infty$.
    Therefore, there exist $C,\omega>0$ such that $\norm{T(t)}\le Ce^{-\omega t}$ for all $t\ge 0$. In particular, 
    \[
        \norm{T(\argument)x}_{L^q([0,t], X)} \le C \left(\int_0^t e^{-\omega s q} \ ds\right)^{1/q} \norm{x} \le M\norm{x}
    \]
    for $M := C (\omega q)^{ -1/q}>0$.
    Moreover, setting $x=0$ in~\eqref{eq:p-q-iss}, we get
    \[
        \norm{\Phi_{(\argument)}u}_{L^q([0,t], X)} \le \gamma(1)
    \]
    whenever $\norm{u}_{L^p([0,t], U)}= 1$. So,
    linearity implies that
    \[
        \norm{\Phi_{(\argument)}u}_{L^q([0,t], X)} \le \gamma(1) \norm{u}_{L^p([0,t], U)}.
    \]
    Setting $G:=\gamma(1)$, and using the triangle inequality, we obtain~(iii). 

    ``(iii) $\Rightarrow$ (ii)'': we have already shown exponential stability above, whereas infinite-time $L^p$-$L^q$-admissibility is immediate by setting $x:=0$ into the inequality in~(iii).
\end{proof}

Proposition~\ref{prop:p-q-iss-implies-exponential-stability} reduces the analysis of $L^p$-$L^q$-ISS for linear systems to that of infinite-time $L^p$-$L^q$ admissibility and exponential stability of the semigroup.
As recalled before, infinite-time $L^p$-admissibility for exponentially stable semigroups is equivalent to $L^p$-admissibility. 
If $p= q$ and the semigroup is analytic, we have shown a corresponding result for $L^p$-$L^p$-admissibility in Theorem~\ref{thm:infinite-admissibility-sufficient}. Consequently, Proposition~\ref{prop:p-q-iss-implies-exponential-stability} can be improved for analytic semigroups in the case $p=q$ as follows:

\begin{corollary}
    \label{cor:p-p-iss-characterisation}
    In Setting~\ref{setting:system}, assume that $(T(t))_{t \ge 0}$ is analytic. Then:
    \begin{center}
        $\Sigma(A,B)$ is $L^p$-$L^p$-ISS for some $p\in [1,\infty)$ \quad $\Leftrightarrow$\quad  $\Sigma(A,B)$ is $L^q$-$L^q$-ISS $\forall \ q\in (1,\infty)$.
    \end{center} Moreover, the following are equivalent for any  $p\in [1,\infty)$.
    \begin{enumerate}[\upshape (i)]
        \item The system $\Sigma(A,B)$ is $L^p$-$L^p$-ISS.
        \item The system $\Sigma(A,B)$ is finite-time $L^p$-$L^p$-admissible and $(T(t))_{t \ge 0}$ is exponentially stable.
        \item The system $\Sigma(A,B)$ is infinite-time $L^p$-$L^p$-admissible and $(T(t))_{t \ge 0}$ is exponentially stable.
    \end{enumerate}
\end{corollary}

\begin{proof}
    The equivalence of (i)-(iii) is a direct combination of Proposition~\ref{prop:p-q-iss-implies-exponential-stability} and Theorem~\ref{thm:infinite-admissibility-sufficient}.

    Next, if $p<\infty$, then $L^p$-$L^p$-ISS of $\Sigma(A,B)$ implies that $\Sigma(A,B)$ is infinite-time $L^p$-$L^p$-admissible and $(T(t))_{t \ge 0}$ is exponentially stable due to Proposition~\ref{prop:p-q-iss-implies-exponential-stability}. As the semigroup is analytic, we infer from Proposition~\ref{prop:p-implies-q} that $\Sigma(A,B)$ is infinite-time  $L^q$-$L^q$-admissible for all $q\in (1,\infty)$. Appealing once again to Proposition~\ref{prop:p-q-iss-implies-exponential-stability}, we obtain that the system must be $L^q$-$L^q$-ISS for all $q\in (1,\infty)$.
\end{proof}

The equivalences in Corollary~\ref{cor:p-p-iss-characterisation} no longer hold in general if $p\ne q$ as noted in Remark~\ref{rem:scalar-system-Counterex} below. Recall from Corollary~\ref{cor:hilbert-analytic-admissible} that if $X$ is a Hilbert space and $(T(t))_{t \ge 0}$ is analytic, then $\Sigma(A,B)$ is automatically finite-time $L^p$-$L^p$-admissible. Consequently, Corollary~\ref{cor:p-p-iss-characterisation} yields the following:

\begin{corollary}
    \label{cor:p-p-iss-hilbert}
    In Setting~\ref{setting:system}, if $X$ is a Hilbert space and $(T(t))_{t \ge 0}$ is analytic, \changeSA{then the following are equivalent:}
    \begin{enumerate}[\upshape (i)]
        \item The system $\Sigma(A,B)$ is $L^p$-$L^p$-ISS for some $p\in [1,\infty)$.
        \item The system $\Sigma(A,B)$ is $L^p$-$L^p$-ISS for all $p\in (1,\infty)$.
        \item The semigroup $(T(t))_{t \ge 0}$ is exponentially stable.
    \end{enumerate}
\end{corollary}

\changeSA{Figure~\ref{fig:ISS-Relations} summarizes the relationships between ISS notions and admissibility.}

\begin{figure}
\centering
\resizebox{\linewidth}{!}{%
\begin{tikzpicture}[admnode]
    \coordinate (A) at (-3.2,0);
    \coordinate (B) at (0,0);
    \coordinate (C) at (4.8,0);
    \coordinate (D) at (4.8,-2.4);
    
    \node (LpISS) at (A) {$L^p$-ISS};
    \node (LpLqISS) at (B) {$L^p$-$L^q$-ISS};
    
    \node (InfLpLq) [text centered,text width=10cm] at (C)
    {infinite-time $L^p$-$L^q$-admissibility\\ +\\ exponential\ stability};
    
    \node (FinLpLq) [text centered,text width=10cm] at (D)
    {finite-time $L^p$-$L^q$-admissibility\\ +\\ exponential\ stability};
    
    \draw[ecvsectioncolor,impl,hyp]
    ($(A)+(0.8,0.2)$) -- ($(B)+(-1,0.2)$);
    
    \node[ecvsectioncolor] at ($($(A)+(0.8,0.2)$)!0.5!($(B)+(-1,0.2)$)+(0,0.35)$)
    {\small $q=\infty$};
    
    \draw[ecvsectioncolor,impl,hyp]
    ($(B)+(1,0.2)$) -- ($(B)+(2,0.2)$);
    
    \draw[implback]
    ($(B)+(1,-0.2)$) -- ($(B)+(2,-0.2)$);
    
    \node[ecvsectioncolor] at ($($(B)+(1,0.2)$)!0.5!($(B)+(2,0.2)$)+(0,0.35)$)
    {\small $q<\infty$};
    
    \draw[ecvsectioncolor,impleq,hyp] (InfLpLq) -- (FinLpLq) node[pos=0.5,left,xshift=10pt] {\small $p=q$} node[pos=0.5,right,xshift=-6pt] {\small analytic};
    
    \draw[counterexample,->,degil]
    ($(B)+(-1,-0.2)$) -- ($(A)+(0.8,-0.2)$);
    \end{tikzpicture}%
}
    \caption{Relations between ISS notions for fixed $p,q\in[1,\infty]$. Solid arrows are unconditional; dotted arrows use the indicated additional assumptions; slashed arrows indicate failure in general.}
    \label{fig:ISS-Relations}
\end{figure}

\section{Lyapunov method}
    \label{sec:lyapunov}

Lyapunov functions are a cornerstone of stability analysis for dynamical and control systems. 
In this section, we develop Lyapunov theory for the analysis of ISS in integral norms.

For a real-valued function $b:\bbR_+\to\R$, recall that the \emph{right-hand upper Dini derivative} at $t\in\bbR_+$ is defined as
\[
    D^+b(t) := \overline{\lim_{h \downarrow 0}} \frac{b(t+h) - b(t)}{h}.
\]
In Setting~\ref{setting:system}, let $x \in X$ and let $V$ be a real-valued function defined in a neighbourhood of $x$. The \emph{Lie derivative 
of $V$ at $x$ corresponding to the input $u$ along the corresponding trajectory of $\Sigma(A,B)$} is defined by
\begin{equation}
\label{eq:Lie-Derivative}
    \dot{V}_u(x):=D^+ V\big(\phi(\cdot,x,u)\big)\Big|_{t=0}=\overline{\lim_{t \downarrow 0}} {\frac{1}{t}\big(V(\phi(t,x,u))-V(x)\big) }.    
\end{equation}
Note that this definition makes sense only if $\phi(\cdot,x,u)$ is defined and lies in $X$ for all sufficiently small times.

\begin{remark}
For the computation of the Dini derivative according to the formula \eqref{eq:Lie-Derivative}, we need to know the solution for a small enough period of time. This makes it possible to develop general theoretical results, but it can be tedious to compute Lie derivatives using this definition in practice. 
\changeSA{It is appealing to check the conditions for strong solutions,
and then argue by density arguments for the resulting
integrated inequalities. For that, one needs, however,
continuity of a function $V$ and suitable continuous dependence
of the flow on initial data and inputs. If $V$ is continuously
differentiable and the flow is absolutely continuous in time,
one can use, for instance, gradient formulas for Lie derivatives
for ordinary differential equations \cite[Section~2.2.3]{Mironchenko2023a}
almost everywhere in time. For delay systems, Driver's
derivative \cite{Pep07} provides another approach under its
corresponding regularity assumptions.} For abstract systems that we consider, mild solutions are in general merely continuous even for $x_0 \in D(A)$ and smooth inputs. \changeSA{However, for analytic systems with initial condition
$x_0\in\dom{A}$, and inputs in
$C^2(\bbR_+,U)$, mild solutions are indeed classical
on $(0,\infty)$; see \cite[discussion prior to Theorem~3]{MiS26}.} \changeSA{If $V$ is continuously Fr\'echet differentiable, classical
derivatives can therefore be computed at positive times.
Integrating the resulting differential inequality and
passing to the limit at zero gives an estimate for the
Dini derivative at zero, provided the right-hand side
is continuous at zero.}
\end{remark}

\subsection{Linear systems with bounded control operators}

The following result gives extensive criteria for $L^p$-$L^q$-ISS, $p\le q$ for the case of bounded input operators. In particular, this provides a converse Lyapunov result for $L^p$-$L^q$-ISS systems with $p\le q$.

\begin{definition}
In Setting~\ref{setting:system}, if $B$ is $\mathrm C$-admissible, then a continuous function $V:X\to \bbR_+$ is called a \emph{coercive $L^1$-ISS Lyapunov function} for $\Sigma(A,B)$ if there there exist $\psi_1,\psi_2 \in \Kinf$, and $a_3,a_4>0$
such that
\[
    \psi_1(\norm{x}) \le V(x) \leq \psi_2(\norm{x}),\qquad x\in X,
\]
and the Lie derivative of $V$ along the trajectories of $\Sigma(A,B)$ satisfies
\[
    \dot{V}_u(x) \leq -a_3V(x) + a_4\|u(0)\|_{U}
\]
for all $x \in X$ and $u\in\mathrm C(\bbR_+,U)$.    
\end{definition}

\begin{theorem}
    \label{thm:p-q-iss-bounded-case}
    In Setting~\ref{setting:system}, if $B\in \calL(U,X)$, then the following statements are equivalent.
    \begin{enumerate}[\upshape (i)]
        \item The system $\Sigma(A,B)$ is $L^p$-ISS for some/all $p \in [1,\infty)$.
        \item The semigroup $(T(t))_{t \ge 0}$ is exponentially stable.
        \item There exists $\lambda>0$ such that the function
            \[
                V(x) := \sup_{t\ge 0} e^{\lambda  t} \norm{T(t)x},\qquad x\in X,
            \]
        is a coercive $L^1$-ISS Lyapunov function for $\Sigma(A,B)$.
        \item The system $\Sigma(A,B)$ is $L^p$-$L^q$-ISS for some/all $p,q \in [1,\infty)$ with $p\le q$.
    \end{enumerate}
\end{theorem}

\begin{proof}
    Since $B$ is bounded, $\Sigma(A,B)$ is $L^p$-admissible for each $p \in [1,\infty]$, so the equivalence of (i) and (ii) holds due to \cite[Theorem~2.1.21]{Mironchenko2023b}. Moreover, the implication  (iii) $\Rightarrow$ (i)  is known from \cite[Theorem~1]{Mironchenko2020}.

    ``(ii) $\Leftrightarrow$ (iv)'': Recall from Proposition~\ref{prop:admissibility-bounded} that if
     $(T(t))_{t \ge 0}$ is exponentially stable and $B$ is bounded, then $\Sigma(A,B)$ is infinite-time $L^p$-$L^q$-admissible
     for all $p,q\in [1,\infty]$ with $p\le q$.
    The equivalence of (ii) and (iv) is therefore a direct consequence of Proposition~\ref{prop:p-q-iss-implies-exponential-stability}.

    ``(ii) $\Rightarrow$ (iii)'': Let $x\in X$. By exponential stability,  there exists $M,\omega>0$ such that $\norm{T(t)}\le Me^{-\omega t}$ for all $t\ge 0$. Fixing $\lambda \in (0,\omega)$, we see that
    \[
        V(x) = \sup_{t\ge 0} e^{\lambda  t} \norm{T(t)x}
    \]
    satisfies $V(x)\ge \norm{x}$ and
    \[
        V(x) \le M\sup_{t\ge 0}e^{(\lambda -\omega)t} \norm{x} \le M \norm{x}.
    \]
    
    Now, since $B\in \calL(U,X)$, the system $\Sigma(A,B)$ is, in particular, $\mathrm C$-admissible. Further, because of exponential stability, \cite[Proposition~7]{MironchenkoWirth2017} guarantees for all $x\in X$ and for all $u\in\mathrm C(\bbR_+, U)$ that 
    \changeSA{\[
    \begin{aligned}
        \dot V_u(x)&\le-\lambda V(x)+V(Bu(0)) \\
                   &\le-\lambda V(x)+M\|B\|\|u(0)\|_U.
    \end{aligned}
\]}
    Therefore, $V$ is a coercive $L^1$-ISS Lyapunov function for $\Sigma(A,B)$.
\end{proof}

\begin{remark}
    \label{rem:scalar-system-Counterex}
    One of the conclusions of Theorem~\ref{thm:p-q-iss-bounded-case} is the fact that for bounded control operators, exponential stability of the semigroup implies $L^p$-$L^q$ ISS if $p\le q$.
    However, as Example~\ref{ex:Scalar-system-Counterex} illustrates (cf. Proposition~\ref{prop:p-q-iss-implies-exponential-stability}), even for scalar systems it does not guarantee $L^p$-$L^q$ ISS for $p>q$, in general.
\end{remark}

\subsection{Lyapunov functions for systems with unbounded control operators}

Next, we turn our attention to unbounded control operators. Unlike the bounded case, here $L^p$-ISS and $L^p$-$L^p$-ISS are no longer equivalent (Example~\ref{ex:diagonal}).

While the classical heat equation with Dirichlet boundary inputs is well-known to be ISS, the existence of coercive ISS Lyapunov functions for this system remains uncertain. Non-coercive Lyapunov functions, introduced in \cite{MironchenkoWirth2019}, expand the applicability of Lyapunov methods to a broader class of systems, and have been successfully applied to the heat equation with Dirichlet boundary inputs as demonstrated in \cite{JacobMironchenkoPartingtonWirth2020}. In the following, we show that non-coercive Lyapunov functions can also be effectively utilised to analyse $L^q$-$L^q $-ISS. Since our systems are linear, we will employ homogeneous Lyapunov functions, a tool widely used for the analysis of non-linear homogeneous systems \cite{Rosier1992}.

\begin{definition}
\label{def:LISS-LF-dissipative}
    In Setting~\ref{setting:system}, let $\Sigma(A,B)$ be $\mathrm C$-admissible. For $q\in [1,\infty)$,
    a continuous function $V:X \to \bbR_+$ is called a \emph{$q$-homogeneous non-coercive ISS Lyapunov function in a dissipative form} for $\Sigma(A,B)$, if $V(ax)=\modulus{a}^qV(x)$ for all $a \in\R$, and there exist
    $a_2, a_3,a_4>0$
    such that
    \begin{equation}
        \label{eq:LyapFunk_1Eig_LISS-dissipative}
        0 < V(x) \leq a_2 \|x\|^q, \qquad x \in X\setminus \{0\},
    \end{equation}
   and the Lie derivative of $V$ along the trajectories of $\Sigma$ satisfies for all $x \in X$ and $u\in\mathrm C(\bbR_+,U)$ the dissipation inequality
    \begin{equation}
        \label{DissipationIneq}
        \dot{V}_u(x) \leq -a_3\|x\|^q + a_4\|u(0)\|^q_{U}.
    \end{equation}
\end{definition}

The following result gives a criterion for $L^q$-$L^q$-ISS for the system $\Sigma(A,B)$ with unbounded control in terms of $q$-homogeneous  ISS Lyapunov functions. Afterwards, we give examples of systems which possess such Lyapunov functions.

\begin{theorem}
    \label{thm:lyapunov-iss-unbounded}
    In Setting~\ref{setting:system},
    let $B\in \calL(U, X_{-1})$ be $\mathrm C$-admissible and assume that the mild solution is continuous with respect to time in the $X$-norm for every $u\in \mathrm C(\bbR_+, U)$.
    
    For $q\in [1,\infty)$, if there exists a non-coercive $q$-homogeneous ISS Lyapunov
    function in a dissipative form for $\Sigma(A,B)$, then $\Sigma(A,B)$ is $L^q$-$L^q$-ISS.
\end{theorem}

\begin{proof} 
    Let $x\in X$ and $u\in \mathrm C(\bbR_+, U)$. Note that~\eqref{DissipationIneq} is an \q{instantaneous} estimate. By $\mathrm C$-admissibility, for a continuous input, the corresponding trajectory exists for all times and lies in $X$, thus \eqref{DissipationIneq}, applied along the trajectory, reads as:
    \begin{equation}
    \label{DissipationIneq_nc-specified}
    \dot{V}_{u(t+\cdot)}\big(\phi(t,x,u)\big) \leq
     -a_3\|\phi(t,x,u)\|^q + a_4\|u(t)\|^q_{U} \qquad (t>0).
    \end{equation}    
    \changeSA{Applying the comparison principle for upper right Dini
derivatives to \eqref{DissipationIneq_nc-specified}, see, e.g., \cite[Proposition~A.31 with continuous $g$]{Mironchenko2023a}, we obtain that} 
    \[
        V\big(\phi(t,x,u)\big) - V(x) \leq -a_3\int_0^t \norm{\phi(s,x,u)}^q ds + a_4\int_0^t \norm{u(s)}^q_{U}ds.
    \]
   Consequently, from the positivity of the Lyapunov function, we get
    \[
        \int_0^t \norm{\phi(s,x,u)}^q \ ds \le \frac{1}{a_3}V(x) + \frac{a_4}{a_3} \int_0^t \norm{u(s)}^q_U\ ds.
    \]
    Therefore,~\eqref{eq:LyapFunk_1Eig_LISS-dissipative} and the inequality $(a+b)^{1/q} \leq (2a)^{1/q} + (2b)^{1/q}$ together imply that
    \[
        \norm{\phi(\argument,x,u)}_{L^q([0,t], X)}
        \le \left(\frac{2a_2}{a_3}\right)^{1/q}\norm{x} + \left(\frac{2a_4}{a_3}\right)^{1/q}\norm{u}_{L^q([0,t],U)}.
    \]
    \changeSA{Finally, fix $t>0$ and $u\in L^q([0,t],U)$.
By the density of continuous functions in $L^q$, choose
$u_n\in C(\bbR_+,U)$ such that
$u_n\to u$ in $L^q([0,t],U)$.
By linearity, applying the above estimate with initial state $0$
to $u_n-u_m$ shows that $\Phi_{(\cdot)}u_n$ is Cauchy in
$L^q([0,t],X)$.
Since $B\in\calL(U,X_{-1})$ and $(T_{-1}(s))_{s\ge0}$
is locally bounded, these trajectories also converge to
$\Phi_{(\cdot)}u$ in $C([0,t],X_{-1})$.
Thus their $L^q([0,t],X)$ limit agrees almost everywhere
with $\Phi_{(\cdot)}u$.
This proves $L^q$-$L^q$-admissibility, and passing to the limit
gives the above estimate for $u$.
As $t>0$ is arbitrary, the above estimate extends to each
$u\in L^q_{\mathrm{loc}}(\bbR_+,U)$.}
    In other words, $\Sigma(A,B)$ is $L^q$-$L^q$-ISS.
\end{proof}

\subsection{Construction of Lyapunov functions}

Having derived direct Lyapunov results, we proceed to the construction of homogeneous  Lyapunov functions. 

\subsubsection*{Lyapunov functions on Hilbert spaces}
Let us start with good news: some converse theorems available for classical ISS can be easily tailored to obtain Lyapunov functions for $L^p$-$L^p$-ISS.
The following result is an immediate corollary of \cite[Theorem 5.3]{JacobMironchenkoPartingtonWirth2020}:

\begin{theorem}
\label{thm:Gen_ISS_LF_Construction}
In Setting~\ref{setting:system}, let $X$ be a complex Hilbert space. 
In addition, let $P\in \calL(X)$ be an operator fulfilling the following conditions:
\begin{enumerate}[\upshape (a)]
    \item For each $x\in X$ with $x\ne 0$, we have
    \[
        \re \duality{Px}{x}_X >0.
    \]

    \item The operator $P$ maps into $\dom{A^*}$, i.e., $\Ima P \subseteq \dom{A^*}$. Here $A^*$ denotes the Hilbert space adjoint of $A$.

    \item The operator $PA$ extends to a bounded linear operator on $X$, which we again denote by $PA$, i.e., $PA\in \calL(X)$.

    \item The operator $P$ satisfies the Lyapunov inequality
    \[
        \re \duality{(PA+A^*P)x  }{x} \le -\duality{x}{x}, 
    \]
    for all $x\in \dom A$.
\end{enumerate}

\changeSA{If $B\in\calL(U,X_{-1})$ is} $L^\infty$-admissible, then the function
\[
    V(x):= \re \duality{Px}{x}_X\qquad,\quad x\in X,
\]
is a $2$-homogeneous non-coercive ISS Lyapunov function in a dissipative form for $\Sigma(A,B)$.
\end{theorem}

\begin{proof}
Clearly, $V$ is $2$-homogeneous and satisfies the sandwich bounds in~\eqref{eq:LyapFunk_1Eig_LISS-dissipative}. 
Furthermore, \cite[Theorem 5.3]{JacobMironchenkoPartingtonWirth2020}
 guarantees that there are $a_1,a_2>0$ such that for all $x\in X$ and all $u \in \changeSA{L^\infty(\bbR_+,U)}$ we have
    \[
    \dot{V}_u(x) \leq -a_1\|x\|^2 + a_2\|u\|^2_{\changeSA{L^\infty(\bbR_+,U)}}.
    \]
    However, since for any $u \in \changeSA{L^\infty(\bbR_+,U)}$, we have  $\dot{V}_u(x) = \dot{V}_v(x)$ for all $v \in \changeSA{L^\infty(\bbR_+,U)}$ with $v=u$ a.e. on $[0,r]$ for some arbitrarily small $r>0$, we obtain for any $r>0$:
    \[
    \dot{V}_u(x) \leq -a_1\|x\|^2 + a_2\esssup_{s\in[0,r]}\changeSA{\|u(s)\|_U^2}.
    \]
    Now taking the limit $r\to 0$ and restricting ourselves to continuous functions only, we obtain for all $x\in X$ and all $u \in\changeSA{\mathrm C(\bbR_+,U)}$ that
    \[
    \dot{V}_u(x) \leq -a_1\|x\|^2 + a_2\changeSA{\|u(0)\|_U^2},
    \]
    which shows that $V$ is indeed a non-coercive $2$-homogeneous ISS Lyapunov function in a dissipative form for $\Sigma(A,B)$ in the sense of this paper.
\end{proof}

\begin{example}[Diagonal systems]
    \label{ex:diagonal}
    On the Hilbert space $X:= \ell^2(\bbN)$, defined in \eqref{eq:ell2}, we consider the semigroup $(T(t))_{t \ge 0}$ generated by
    \[
        Ae_n = -\lambda_n e_n \quad \text{with}\quad \dom{A} := \left\{ (x_n)\in X: (\lambda_n x_n) \in X\right\},
    \]
    where $(e_n)$ denotes the canonical orthonormal basis of $X$ and
    $(\lambda_n)$ is a strictly increasing sequence in $(0,\infty)$ that diverges to $\infty$. For $B\in \changeSA{\calL(\bbC, X_{-1})}$, the function
    \[
        V(x) := \sum_{n\in \bbN} \frac{1}{\lambda_n} \changeSA{\modulus{\duality{x}{e_n}}^2},\quad x \in X,
    \]
    where $\duality{\argument}{\argument}$ is the usual inner product on $X$, is a non-coercive $2$-homogeneous ISS Lyapunov function in a dissipative form for $\Sigma(A,B)$ and $\Sigma(A,B)$ is $L^p$-$L^p$-ISS for all $p\in (1,\infty)$.
    
        Indeed, $V(ax)=a^2V(x)$ for all $a\in \bbR$ and from \cite[Proposition~4.6.9]{Mironchenko2023b}, it follows that the assumptions of Theorem~\ref{thm:Gen_ISS_LF_Construction} are fulfilled, whence Theorem~\ref{thm:Gen_ISS_LF_Construction} guarantees that 
    $V$ is a non-coercive $2$-homogeneous ISS Lyapunov function in a dissipative form for $\Sigma(A,B)$.

    Furthermore, as $A$ generates a bounded analytic semigroup
    of contractions on a Hilbert space (cf. \cite[Proposition~6.9]{JacobMironchenkoPartingtonWirth2020}), so each $B\in \changeSA{\calL(\bbC, X_{-1})}$ is even $L^\infty$-admissible and we have continuity of the mild solution; see \cite[Theorem~1]{JacobSchwenningerZwart2019}. Thus, $\Sigma(A,B)$ is $L^2$-$L^2$-ISS by Theorem~\ref{thm:lyapunov-iss-unbounded}. Due to analyticity of the semigroup, we can apply Propositions~\ref{prop:p-implies-q} and~\ref{prop:p-q-iss-implies-exponential-stability}, to even get $L^p$-$L^p$-ISS for all $p\in (1,\infty)$.
\end{example}

We point out that in Example~\ref{ex:diagonal} while $\Sigma(A,B)$ is $L^p$-$L^p$-ISS for each $B\in \changeSA{\calL(\bbC, X_{-1})}$ and all $p\in (1,\infty)$, there exist $B\in \changeSA{\calL(\bbC, X_{-1})}$ that are not $L^p$-admissible for any $p<\infty$ \cite[Example~5.2]{JacobNabiullinPartingtonSchwenninger2018}, and in turn, the corresponding system $\Sigma(A,B)$ cannot be $L^p$-ISS.

\begin{example}[Heat equation with Dirichlet boundary input]
    Consider the boundary control system
    \[
        \begin{aligned}
            x_t(\xi,t)& =a x_{\xi \xi}(\xi,t), \qquad && \xi\in(0,1),~ t>0,\\
            x(0,t)&=0,  \quad x(1,t)=u(t), \qquad && t>0,\\
            x(\xi,0)&=x_{0}(\xi)            \qquad && \xi \in (0,1),
        \end{aligned}
    \]
    where $a>0$. Choosing $X=L^{2}(0,1)$ and $U=\mathbb C$, the above system can equivalently be written as a system $\Sigma(A,B)$ with 
    \begin{align*}
        Af = a f''
        \quad\text{with}\quad
        \dom{A} = \{x \in H^2(0,1) : x(0) = x(1)=0\},
    \end{align*}
    the Dirichlet Laplacian and a certain $B\in \calL(U, X_{-1})$; see, for instance \cite[Proposition~10.1.2 and Remark~10.1.4]{TucsnakWeiss2009}. Observe that the function 
    \[
         V(x) := -  \langle A^{-1}x,x\rangle_X 
                 = \int_0^1 \left(\int_\xi^1  (\xi -\tau) x(\tau) d\tau \right)\overline{x(\xi)}\ d\xi,
    \]
    of course, satisfies, $V(ax)=\modulus{a}^2 V(x)$ for all $a\in \bbR$.
    Therefore, from \cite[Example 6.10]{JacobMironchenkoPartingtonWirth2020}, it follows that
    the assumptions of Theorem~\ref{thm:Gen_ISS_LF_Construction} are fulfilled, and in turn,
    $V$ is a non-coercive $2$-homogeneous ISS Lyapunov function in a dissipative form for $\Sigma(A,B)$. 
    
    It is well-known that the Dirichlet Laplacian generates an exponentially stable analytic contraction semigroup on $X$. Repeating the arguments at the end of Example~\ref{ex:diagonal} and applying Theorem~\ref{thm:infinite-admissibility-sufficient}, it follows that $\Sigma(A,B)$ is $L^p$-$L^p$-ISS for all $p\in (1,\infty)$.
\end{example}

\subsubsection*{Lyapunov functions corresponding to analytic systems}

We close our analysis by constructing a Lyapunov function for analytic systems over Banach spaces. Recall that for the system $\Sigma(A,B)$ in Setting~\ref{setting:system}, the control operator $B$ a priori maps into $X_{-1}$. For what follows, we assume, in addition, that 
$(T(t))_{t\ge 0}$ is analytic and exponentially stable. Then the fractional Sobolev spaces $X_{-\alpha}$ for $\alpha \in (0,1)$ can be used to obtain a criterion for admissibility of $\Sigma(A,B)$, see for instance, \cite{MaraghBounitFadiliHammouri2014}. Recall that if $(T(t))_{t\ge 0}$ is exponentially stable and analytic, then $(T(t))_{t\ge 0}$ is bounded analytic and hence $A$ is sectorial \cite[Theorem~II.4.6]{EngelNagel2000} and for $\alpha>0$, the fractional powers $A^{-\alpha}:X \to X$ are bounded injective operators defined by
\[
    A^{-\alpha} := \frac{1}{2\pi i}\int_{\gamma} \lambda^{-\alpha} \Res(\lambda, A)\ d\lambda,
\]
where $\gamma$ is the boundary of a sector (arising from analyticity of the semigroup) that contains $\spec(A)$; see \cite[Section~II.5.c]{EngelNagel2000}.  
Injectivity allows us to define $A^{\alpha}:= \left(A^{-\alpha}\right)^{-1}$ with $\dom{A^{\alpha}} = \Ima A^{-\alpha}$. In particular, $ A^{-\alpha}A^{\alpha}$ acts as the identity operator on $\Ima A^{-\alpha}$. Analogous to the integer case, the \emph{fractional extrapolation space} -- also known as \emph{abstract Sobolev space} -- $X_{-\alpha}$ is defined as the completion of $X$ with respect to the norm $\norm{x}_{-\alpha}:=\norm{A^{-\alpha}x}_X$. This gives rise to the natural chain of continuous inclusions:
\[
    X \subset X_{-\alpha} \subset X_{-1}
\]
whenever $\alpha \in (0,1)$. Additionally, $A^{-\alpha}$ extends uniquely to an isometric isomorphism from $X_{-\alpha}$ to $X$, which we again denote by $A^{-\alpha}$. Various analyticity properties of the generator carry over to fractional powers, for example, $\Ima T(t)\subseteq \dom{A^\alpha}$ for all $t>0$ and
\begin{equation}
    \label{eq:fractional-analytic}
    \sup_{t>0} t^{\alpha } e^{\omega t} \norm{T(t)A^{\alpha} } <\infty \quad \text{for some }\omega>0.
\end{equation}
For a more detailed discussion about fractional powers and abstract Sobolev spaces, we refer to \cite{Haase2006} and \cite[Section~II.5.c]{EngelNagel2000}.

The proof of our converse Lyapunov result requires some technical considerations, which we outsource to the following lemma:
\begin{lemma}
    \label{lem:converse-lyapunov}
    In Setting~\ref{setting:system}, assume that $(T(t))_{t\ge 0}$ is exponentially stable and analytic and $B\in \calL(U, X_{-\alpha})$ for some $\alpha \in [0,1)$. Then $B$ is $\mathrm C$-admissible, and the mild solution given by~\eqref{eq:flow} is continuous with values in $X$. Additionally, there exist $C,\omega>0$ such that
    \begin{equation}
        \label{eq:lem:upper-estimate}
            \norm{T(t)\Phi_hu} \le  C t^{-\alpha } e^{-\omega  t}\norm{ A^{-\alpha}\Phi_h u}
    \end{equation}
    for all $t> 0$, $h\in (0,1)$, and $u\in \mathrm C(\bbR_+, U)$.
    In particular,
    \begin{equation}
        \label{eq:lem:limit-semigroup}
        \overline{\lim_{h \downarrow 0}}\frac{\norm{T(t)\Phi_hu}}{h} \le C\norm{A^{-\alpha}Bu(0)}t^{-\alpha}e^{-\omega t}
    \end{equation}
    for all $u\in \mathrm C(\bbR_+, U)$ and $t>0$.
\end{lemma}

\begin{proof}
    First of all, for exponentially stable analytic semigroups, the assumption $B\in \calL(U, X_{-\alpha})$ implies by \cite[Proposition~2.13(ii)]{Schwenninger2020} that $B$ is $L^p$-admissible for all $p>(1-\alpha)^{-1}$. In particular, $B$ is $\mathrm C$-admissible, and by \cite[Proposition~2.3]{Weiss1989b} the mild solution given by~\eqref{eq:flow} is continuous with values in $X$. 
    
    Next, from the boundedness of $A^{-\alpha}$ and the fact that it commutes with the semigroup operators, for each $h>0$, we have
    \[
        A^{-\alpha}\Phi_h u= \int_0^h T(h-s)A^{-\alpha}Bu(s)\ ds.
    \]
    Whence, since $A^{-\alpha }B\in \calL(U,X)$, we have
    \begin{equation}
        \label{eq:lem:limit-generator}
        \lim_{h\to 0} \frac{A^{-\alpha}\Phi_h u}{h} = A^{-\alpha }Bu(0).
    \end{equation}
    On the other hand, due to~\eqref{eq:fractional-analytic}, we can find $C,\omega>0$ such that for each $t,h>0$
    \begin{align*}
        \norm{T(t)\Phi_hu} & = \norm{T(t)A^\alpha A^{-\alpha}\Phi_h u}\\
                             & \le \norm{T(t)A^\alpha} \norm{ A^{-\alpha}\Phi_h u}\\
                             & \le Ct^{-\alpha } e^{-\omega  t}\norm{ A^{-\alpha}\Phi_h u}.
    \end{align*}
    The estimate~\eqref{eq:lem:limit-semigroup} can now be inferred from this inequality and~\eqref{eq:lem:limit-generator}.
\end{proof}

Most converse Lyapunov results for linear systems with unbounded input operators are confined to quadratic Lyapunov functions. In contrast to that, next, we present a converse non-coercive ISS Lyapunov theorem for linear systems with unbounded input operators, guaranteeing the existence of non-quadratic Lyapunov functions:

\begin{theorem}
    \label{thm:converse-lyapunov}
    In Setting~\ref{setting:system}, assume that $(T(t))_{t\ge 0}$ is exponentially stable and analytic and $B\in \calL(U, X_{-\alpha})$ for some $\alpha \in [0,1)$. \changeSA{Then for each $n\in\bbN$ with $n\alpha<1$,} the function
\begin{align}
\label{eq:LF-n-homogeneous}
        V(x):= \int_0^\infty \norm{T(t)x}^n\ dt    
\end{align}
    is a non-coercive $n$-homogeneous  locally Lipschitz continuous ISS Lyapunov function for $\Sigma(A,B)$. In particular, $\Sigma(A,B)$ is $L^q$-$L^q$-ISS for each $q\in [1,\infty)$.
\end{theorem}

\begin{proof}
    Due to Lemma~\ref{lem:converse-lyapunov}, we have $\mathrm C$-admissibility of the control operator and continuity of the mild solution. 
    Since $\alpha\in[0,1)$, we have \changeSA{$1\cdot\alpha<1$}, and thus the first statement of the theorem shows that there is a non-coercive $1$-homogeneous  ISS Lyapunov function.
    Hence, Theorem~\ref{thm:lyapunov-iss-unbounded} ensures that $\Sigma(A,B)$ is $L^1$-$L^1$-ISS. 
    Since $(T(t))_{t \ge 0}$ is analytic, Corollary~\ref{cor:p-p-iss-characterisation} ensures that $\Sigma(A,B)$ is $L^q$-$L^q$-ISS for each $q\in [1,\infty)$.
    In other words, the second assertion is a consequence of the first.

    To prove the claims on $V$, we begin by noting that the  $n$-homogeneity of $V$ is immediate, and its local Lipschitz continuity is a straightforward consequence of the exponential stability of the semigroup.
    
    In addition, $V(x)>0$ for all $x\in X\setminus \{0\}$ and by exponential stability
    \[
        V(x)\le M^n \norm{x}^n \int_0^\infty e^{-\omega n t}\ dt = \frac{M^n}{n\omega} \norm{x}^n
    \]
    for some $M,\omega>0$. In particular, $V$ satisfies the non-coercive sandwich inequality~\eqref{eq:LyapFunk_1Eig_LISS-dissipative} with $a_2:=M^n(n\omega)^{-1}$.

    Next, let us verify the dissipation inequality. 
    Consider arbitrary but fixed $h\in (0,1)$, a function $u\in \mathrm C(\bbR_+, U)$, and $x\in X$. Then
    $\Phi_h u\in X$ as $B$ is $\mathrm C$-admissible. We split the estimation of the Lie derivative into two cases:

    \emph{Case 1: $n=1$}. First, note that,
    \begin{align*}
        V\big(\phi(h,x,u)\big) - V(x) & = V(T(h)x+\Phi_h u) - V(x)\\
                                      & \le \int_0^\infty \norm{T(t+h)x}\ dt-\int_0^\infty \norm{T(t)x}\ dt + \int_0^\infty\norm{T(t)\Phi_h u}\ dt\\
                                      & =\int_h^\infty \norm{T(t)x}\ dt-\int_0^\infty \norm{T(t)x}\ dt + \int_0^\infty\norm{T(t)\Phi_h u}\ dt\\
                                      &=-\int_0^h \norm{T(t)x}\ dt + \int_0^\infty\norm{T(t)\Phi_h u}\ dt.
    \end{align*}  
    \changeSA{By \eqref{eq:lem:upper-estimate} and the identity for $A^{-\alpha}\Phi_hu$ in the
proof of Lemma~\ref{lem:converse-lyapunov}, we have
\[
    \frac{\|T(t)\Phi_hu\|}{h}
    \le MC\|A^{-\alpha}B\|\|u\|_{L^\infty([0,1],U)}
     t^{-\alpha}e^{-\omega t}
\]
for all $t>0$ and $h\in(0,1)$, where the constants are
chosen so that the exponential stability bound and \eqref{eq:lem:upper-estimate}
hold with the same $\omega>0$.
Since $\alpha<1$, the right-hand side is integrable.
Now using \eqref{eq:lem:limit-semigroup} and applying the reverse Fatou lemma,
we obtain
\[
    \begin{aligned}
        \limsup_{h\downarrow0}\frac1h
        \int_0^\infty\|T(t)\Phi_hu\|\,dt
        &\le C\|A^{-\alpha}Bu(0)\|
        \int_0^\infty t^{-\alpha}e^{-\omega t}\,dt\\
        &\le C_1\|u(0)\|_U
    \end{aligned}
\]
with $C_1:=C\omega^{\alpha-1}\Gamma(1-\alpha)
\|A^{-\alpha}B\|$, where $\Gamma$ denotes a gamma-function.} As a result, 
    \begin{align*}
        \dot{V}_u(x) = \overline{\lim_{h \downarrow 0}}\frac{V\big(\phi(h,x,u)\big) - V(x)}{h} 
                     \le -\norm{x} +C_1 \norm{ u(0)}_U,
    \end{align*}
    where we have used that $\overline{\lim_{h \downarrow 0}}h^{-1}\int_0^h \norm{T(t)x}\ dt=\norm{x}$. In particular, the dissipation inequality~\eqref{DissipationIneq} is fulfilled with \changeSA{$a_3=1$ and any $a_4\ge C_1$}.
    
    \emph{Case 2: $n>1$}. In this case, we adapt the proof of \cite[Proposition~6]{MironchenkoWirth2017}. First, using the binomial inequality, we obtain that
    \begin{align*}
        V\big(\phi(h,x,u)\big)  & = V\big(T(h)x+\Phi_h u\big)\\
                                      & \le \int_0^\infty \big[\norm{T(t+h)x} +\norm{T(t)\Phi_h u}\big]^n\ dt\\
                                      & = \int_0^\infty \norm{T(t+h)x}^n\ dt+\int_0^\infty\norm{T(t)\Phi_h u}^n\ dt\\
                                      &\qquad\qquad+\sum_{k=1}^{n-1} \binom{n}{k} \int_0^\infty \norm{T(t+h)x}^k \norm{T(t)\Phi_hu}^{n-k}\ dt\\
                                      & = \int_h^\infty \norm{T(t)x}^n\ dt+\int_0^\infty\norm{T(t)\Phi_h u}^n\ dt\\
                                      &\qquad\qquad+\sum_{k=1}^{n-1} \binom{n}{k} \int_0^\infty \norm{T(t+h)x}^k \norm{T(t)\Phi_hu}^{n-k}\ dt.
    \end{align*}
    Therefore, the Lie derivative can be estimated as
    \begin{align*}
            \dot{V}_u(x)
                    & = \overline{\lim_{h \downarrow 0}}\frac{V\big(\phi(h,x,u)\big) - V(x)}{h} \\
                    & \le -\overline{\lim_{h \downarrow 0}}\frac1h\int_0^h \norm{T(t)x}^n\ dt+\overline{\lim_{h \downarrow 0}}\frac1h\int_0^\infty\norm{T(t)\Phi_h u}^n\ dt\\
                    &\qquad\qquad+\overline{\lim_{h \downarrow 0}}\frac1h\sum_{k=1}^{n-1} \binom{n}{k} \int_0^\infty \norm{T(t+h)x}^k \norm{T(t)\Phi_hu}^{n-k}\ dt
    \end{align*}
    which gives that
    \begin{equation}
        \label{eq:thm:lie-derivative}
        \dot{V}_u(x) \le -\norm{x}^n +I_1+I_2;
    \end{equation}
    where 
    $
        I_1 := \overline{\lim_{h \downarrow 0}}\frac1h\int_0^\infty\norm{T(t)\Phi_h u}^n\ dt
    $
    and
    \[
        I_2 : = \overline{\lim_{h \downarrow 0}}\frac1h\sum_{k=1}^{n-1} \binom{n}{k} \int_0^\infty \norm{T(t+h)x}^k \norm{T(t)\Phi_hu}^{n-k}\ dt.
    \]
    First of all, $I_1=0$ because for $t>0$, we habe
    \[
        \norm{T(t)\Phi_h u}^n \le  (MC)^n t^{-n\alpha}e^{-n\omega t}\norm{A^{-\alpha}B}^n \norm{u}^n_{L^\infty([0,1],U)}h^n
    \]
    by~\eqref{eq:lem:upper-estimate} and so  $I_1$ must vanish as $n>1$.

    To estimate $I_2$ from above, we again use~\eqref{eq:lem:upper-estimate} to note for each $1\le k\le n-1$ that
    \begin{align*}
        \norm{T(t+h)x}^k \norm{T(t)\Phi_hu}^{n-k} \le K \norm{x}^k \norm{u}^{n-k}_{L^\infty([0,1],U)} \changeSA{h^{n-k}} t^{-(n-k)\alpha}e^{-(n-k)\omega t}
    \end{align*}
    with $K:=C^{n-k} M^n \norm{A^{-\alpha}B}^{n-k}$
    for all $t>0$ \changeSA{and $h\in(0,1)$}.
    Since $(n-k)\alpha<1$ and $\omega>0$, the function $t\mapsto t^{-(n-k)\alpha}e^{-(n-k)\omega t}$ lies in $L^1(0,\infty)$. \changeSA{Together with $h^{n-k-1}\le1$, this enables us to employ
the reverse Fatou lemma and \eqref{eq:lem:limit-semigroup}, followed by dominated
convergence theorem to obtain for each $\epsilon\in(0,1)$ that}
    \begin{align*}
        I_2 &\le \overline{\lim_{h \downarrow 0}}\sum_{k=1}^{n-1}\binom{n}{k}\int_0^\infty \norm{T(t+h)x}^k C^{n-k}\norm{A^{-\alpha}Bu(0)}^{n-k} t^{-(n-k)\alpha}e^{-(n-k)\omega t}\ dt\\
        &= \overline{\lim_{h \downarrow 0}}\sum_{k=1}^{n-1}\binom{n}{k}\int_0^\infty \epsilon^{k}\norm{T(t+h)x}^k \frac{C^{n-k}}{\epsilon^{k}}\norm{A^{-\alpha}Bu(0)}^{n-k} t^{-(n-k)\alpha}e^{-(n-k)\omega t}\ dt\\
        &\le 2^{n-1} \overline{\lim_{h \downarrow 0}}\sum_{k=1}^{n-1} \int_0^\infty \big[\epsilon^n\norm{T(t+h)x}^n + C^n \epsilon^{-\frac{kn}{n-k}} \norm{A^{-\alpha}Bu(0)}^{n} t^{-n\alpha}e^{-n\omega t}\big]\ dt\\
        &\le 2^{n-1}\overline{\lim_{h \downarrow 0}} \sum_{k=1}^{n-1} \int_0^\infty \big[\epsilon^n\norm{T(t+h)x}^n + C^n \changeSA{\epsilon^{-n(n-1)}} \norm{A^{-\alpha}Bu(0)}^{n} t^{-n\alpha}e^{-n\omega t}\big]\ dt\\
    &\le 2^{n-1}(n-1) \left[\frac{\epsilon^n M^n}{n\omega}\norm{x}^n +\changeSA{\frac{C^n\Gamma(1-n\alpha)}{\epsilon^{n(n-1)}(n\omega)^{1-n\alpha}}} \norm{A^{-\alpha}B}^n\norm{u(0)}_U^{n} \right];
    \end{align*}
    \changeSA{where the second and third inequalities use
\[
    \binom nk a^kb^{n-k}\le(a+b)^n\le2^{n-1}(a^n+b^n),
    \qquad a,b\ge0,
\]
and
\[
    \epsilon^{-kn/(n-k)}\le\epsilon^{-n(n-1)},
    \qquad 1\le k\le n-1,\quad 0<\epsilon<1,
\]
respectively, and by substitution we have
$\int_0^\infty t^{-n\alpha}e^{-n\omega t}\,dt
=\Gamma(1-n\alpha)(n\omega)^{n\alpha-1}$.}
    
    Choosing $\epsilon$ such that
    \[
        a_3:=1-\frac{2^{n-1}(n-1)\epsilon^n M^n}{n\omega} >0
    \]
    and substituting in~\eqref{eq:thm:lie-derivative}, we get
    \[
        \dot{V}_u(x) \le -a_3\norm{x}^n +a_4 \norm{u(0)}_U^n
    \]
    for some $a_4>0$, again verifying the dissipation inequality~\eqref{DissipationIneq}.
\end{proof}

\begin{remark}
    It is interesting that for $B\in \calL(U, X_{-\alpha})$ with any $\alpha \in [0,1)$, we can construct a 1-homogeneous non-coercive ISS Lyapunov function. \changeSA{But to apply Theorem~\ref{thm:converse-lyapunov} with $n=2$, we require
$B\in\calL(U,X_{-\alpha})$ for some $\alpha<\frac12$
(which is stronger than $L^2$-admissibility).}     
\end{remark}

\section*{Acknowledgments}
The authors thank Felix Schwenninger for various fruitful discussions.

\bibliographystyle{plainurl}
\bibliography{literature}
\end{document}